\documentclass[a4paper,11pt]{article}
\usepackage{makeidx}
\usepackage{graphicx}
\usepackage[english,activeacute]{babel}
\usepackage[latin1]{inputenc}
\usepackage{amsmath}
\usepackage{amsfonts}
\usepackage{amssymb,latexsym}
\usepackage{colortbl}
\usepackage{multicol}
\usepackage{hyperref}

\newtheorem{theorem}{Theorem}

\newtheorem{corollary}{Corollary}

\newtheorem{lemma}{Lemma}

\newtheorem{proposition}{Proposition}
\newtheorem{remark}{Remark}

\newenvironment{proof}[1][Proof]{\noindent\textbf{#1.} }{\ \rule{0.5em}{0.5em}}
\begin{document}

\title{New analytic properties of nonstandard Sobolev-type Charlier
orthogonal polynomials}
\author{Edmundo J. Huertas$^{1}$, Anier Soria-Lorente$^{2}$ \\
$^{1}$Departamento de Ingeniería Civil: Hidráulica y Ordenación del
Territorio\\
E.T.S. de Ingeniería Civil, Universidad Politécnica de Madrid\\
C/ Alfonso XII, 3 y 5, 28014 Madrid, Spain.\\
ej.huertas.cejudo@upm.es, ehuertasce@gmail.com\\
\\
$^{2}$Department of Basic Sciences, Granma University\\
Km. 17.5 de la carretera de Bayamo-Manzanillo, Bayamo, Cuba\\
asorial@udg.co.cu}
\date{\emph{(\today)}}
\maketitle

\begin{abstract}
In this contribution we consider the sequence $\{Q_{n}^{\lambda }\}_{n\geq
0} $ of monic polynomials orthogonal with respect to the following inner
product involving differences 
\begin{equation*}
\langle p,q\rangle _{\lambda }=\int_{0}^{\infty }p\left( x\right)
q\left(x\right) d\psi ^{(a)}(x)+\lambda \,\Delta p(c)\Delta q(c),
\end{equation*}
where $\lambda \in \mathbb{R}_{+}$, $\Delta $ denotes the forward difference
operator defined by $\Delta f\left( x\right) =f\left( x+1\right)
-f\left(x\right) $, $\psi ^{(a)}$ with $a>0$ is the well known Poisson
distribution of probability theory%
\begin{equation*}
d\psi ^{(a)}(x)=\frac{e^{-a}a^{x}}{x!}\quad \text{at }x=0,1,2,\ldots ,
\end{equation*}%
and $c\in \mathbb{R}$ is such that $\psi ^{(a)}$ has no points of increase
in the interval $(c,c+1)$. We derive its corresponding hypergeometric
representation. The ladder operators and two different versions of the
linear difference equation of second order corresponding to these
polynomials are given. Recurrence formulas of five and three terms, the
latter with rational coefficients, are presented. Moreover, for real values
of $c$ such that $c+1<0$, we obtain some results on the distribution of its
zeros as decreasing functions of $\lambda $, when this parameter goes from
zero to infinity.

\medskip


\medskip


\medskip

\end{abstract}



\section{Introduction}

\label{[S1]-Intro}



Let $\mathbb{P}$ be the linear space of polynomials with real coefficients.
In this contribution we analyze the sequence $\{Q_{n}^{\lambda }\}_{n\geq 0}$
of monic polynomials orthogonal with respect to the following inner product
on $\mathbb{P}$ involving differences%
\begin{equation}
\langle p,q\rangle _{\lambda }=\int_{0}^{\infty }p\left( x\right) q\left(
x\right) d\psi ^{(a)}(x)+\lambda \,\Delta p(c)\Delta q(c),
\label{S1-SobtypInnPr}
\end{equation}%
where $\lambda \in \mathbb{R}_{+}$, $\Delta $ denotes the forward difference
operator defined by $\Delta f\left( x\right) =f\left( x+1\right) -f\left(
x\right) $. Concerning $d\psi ^{(a)}$ we consider the well known Poisson
distribution of probability theory where $a>0$, and $\psi ^{(a)}(x)$ is the
step function with the jump%
\begin{equation}
\frac{e^{-a}a^{x}}{x!}\quad \text{at the point }x=0,1,2,\ldots .
\label{S1-CharlierMeasure}
\end{equation}

All the results presented in this paper (except those in Section \ref%
{[S6]-Zeros}, where more restrictive conditions on $c$ will be needed) are
valid whenever $c\in \mathbb{R}$ and $\psi ^{(a)}$ has no points of increase
in the interval $(c,c+1)$. As was already explained in \cite{B-JCAM95}, for
this particular case it means that $c\in \mathbb{R}$ can be chosen as
precisely one of the points of the spectrum of $\psi ^{(a)}$, or even it can
be chosen in such a way that, if the spectrum of $\psi ^{(a)}$ is contained
in an interval $I$, the condition $I\cap (c,c+1)=\varnothing $ is verified.
These restrictions on the values of $c$ come from the fact that we require
the distribution $(x-c)(x-c-1)d\psi ^{(a)}$ to be nonnegative on the
positive real semiaxis.

The study of orthogonal polynomials with respect to inner products involving
differences as (\ref{S1-SobtypInnPr}) was introduced by H. Bavinck\ in a
series of seminal papers \cite{B-JCAM95}, \cite{B-AA95} and \cite{B-IM96},
by analogy with the so called \textit{discrete Sobolev inner products}
involving derivatives (see for example \cite{AMRR-SJMA92}, \cite{MPP-RdM92}, 
\cite{M-JCAM93} and the recent survey \cite{MX-EM14}). Because of this
analogy, we follow \cite{B-IM96} in calling the elements of the sequence of
monic orthogonal polynomials (SMOP in the sequel) $\{Q_{n}^{\lambda
}\}_{n\geq 0}$ to be \textit{Sobolev-type Charlier} orthogonal polynomials.
Conjointly, in those three said primary papers is given a general explicit
representation of $Q_{n}^{\lambda }(x)$, not only valid on the Charlier
case, but also useful for any other discrete orthogonality measure. It is
proved that $\{Q_{n}^{\lambda }\}_{n\geq 0}$ satisfy a five term recurrence
relation and an analogue of the Christoffel-Darboux formula is presented,
along with several results on their corresponding zero distribution.
Moreover, in the last of the three papers it is shown that these
Sobolev-type Charlier orthogonal polynomials are eigenfunctions of an 
\textit{infinite order difference operator} which, together with the
corresponding eigenvalues, are both linear perturbations of those in the
classical Charlier case. The special case $c=0$ is deeply analyzed
throughout all these three works.

Since then, and to the best of our knowledge, the Charlier case has remained
untouched. Several researchers have done further work on the Sobolev-type
case for discrete orthogonality measures, but mainly concerning the Meixner
case (see \cite{AGM-JDEA00}, \cite{AGMB-JCAM00}, \cite{KD-AAS12}, \cite%
{M-JCAM15}, \cite{MPP-RJ11} and the references given there).

It must be clear that the kind of modification studied here is of quite a
different nature to that considered in \cite{AGM-JCAM95}, where are also
considered discrete point masses in the inner product. The kind of
modifications studied there give rise to new families of \textit{standard}
orthogonal polynomials. They are standard, in the sense that the operator of
multiplication by $x$ is symmetric with respect to such an inner product,
i.e. $\langle xp,q\rangle =\langle p,xq\rangle $, and then the well known
nice properties of standard orthogonal polynomials arise: there exist a
three term recurrence relation, the zeros of consecutive polynomials in the
sequence interlace, all the zeros are simple and real, and a long etcetera.
Quite the opposite, the SMOP $\{Q_{n}^{\lambda }\}_{n\geq 0}$\ considered
here is not standard in the aforementioned way, and we will not find those
nice properties anymore. It means that can be complex zeros, there is no
three term recurrence relation in the usual way, etc, etc, and therefore $%
\{Q_{n}^{\lambda }\}_{n\geq 0}$ is known as a \textit{non-standard} sequence.

In this paper, we delve once more into the inner product (\ref%
{S1-SobtypInnPr}), applying some recent computational and analytical
techniques to obtain fresh results for the Sobolev-type Charlier polynomials 
$\{Q_{n}^{\lambda }\}_{n\geq 0}$ and their zeros. First, we establish its $%
_{3}F_{1}$ hypergeometric character, which was unknown so far. Next, we
obtain explicit expressions for the \textit{ladder difference operators} and
we use them to obtain two different versions of the \textit{second order
difference equation} satisfied by the Sobolev-type Charlier SMOP. We also
use the ladder difference operators to obtain a kind of three term
recurrence formula with rational coefficients, which allows us to find every
polynomial $Q_{n+1}^{\lambda }(x)$ of precise degree $n+1$, in terms of only
the previous two consecutive polynomials of the SMOP $Q_{n}^{\lambda }(x)$
and $Q_{n-1}^{\lambda }(x)$. In \cite{B-AA95} it was proved that this
Sobolev-type Charlier SMOP satisfy a five term recurrence relation, and here
we provide the explicit expression for the coefficients in this high order
recurrence formula. Finally, we find a new representation for the
Sobolev-type Charlier orthogonal polynomials which is useful to obtain sharp
limits (and the speed of convergence to them) of their zeros in terms of the
parameter\ of the perturbation $\lambda $, which somehow determines how
important the Sobolev-type perturbation is on the classical Charlier measure 
$\psi ^{(a)}$.

The structure of the manuscript is as follows. In the next Section we recall
some basic facts of the classical Charlier polynomials, which will be needed
in the sequel. In Section \ref{[S3]-ConnForms} we deal with some connection
formulas and the hypergeometric representation of $Q_{n}^{\lambda }(x)$. In
Section \ref{[S4]-Ladder}, proceeding directly from the representation of
the SMOP $\{Q_{n}^{\lambda }\}_{n\geq 0}$\ given in \cite[eq. (2.13)]%
{B-JCAM95}, we provide the difference ladder operators and the second linear
difference equation that they satisfy. Section \ref{[S5]-5TRR} is devoted to
the previously mentioned fundamental recurrence formulas for this
Sobolev-type Charlier SMOP. Finally, Section \ref{[S6]-Zeros}, is focused on
the behavior of the zeros of these polynomials in terms of the mass
parameter $\lambda $.



\section{Preliminaries}

\label{[S2]-Prelim}



The forward ($\Delta $) and backward ($\nabla $) difference operators are
defined by%
\begin{equation*}
\Delta f\left( x\right) =f\left( x+1\right) -f\left( x\right) ,\qquad \nabla
f(x)=f(x)-f(x-1).
\end{equation*}%
These operators satisfy the following properties, which will be useful in
the sequel%
\begin{equation}
\begin{array}{rcl}
\Delta \left[ f(x)g(x)\right] & = & f(x)\Delta g(x)+g(x)\Delta f(x)+\Delta
f(x)\Delta g(x), \\ 
\nabla \left[ f(x)g(x)\right] & = & f(x)\nabla g(x)+g(x-1)\nabla f(x)%
\end{array}
\label{S2-RegProd}
\end{equation}

\begin{equation}
\Delta \left[ \frac{f(x)}{g(x)}\right] =\frac{g(x)\Delta f(x)-f(x)\Delta g(x)%
}{g(x)\left[ g(x)+\Delta g(x)\right] }.  \label{S2-QuoRule}
\end{equation}

Let $\{C_{n}^{(a)}\}_{n\geq 0}$ be the sequence of monic Charlier
polynomials, orthogonal with respect to the inner product on $\mathbb{P}$
(see for instance \cite[Ch. VI.1]{Chi78}, \cite[Ch. 6]{Ism05}, \cite[Ch. 9]%
{Koekoek10}, \cite[Ch. 2]{Nikiforov91}, \cite[Section 6.7]{Szego75}, and the
references therein)%
\begin{equation*}
\left\langle p,q\right\rangle =\int_{0}^{\infty }p(x)q(x)d\psi
^{(a)}(x),\quad a>0,
\end{equation*}%
which can be explicitly given in the following several equivalent ways%
\begin{equation}
C_{n}^{(a)}(x)=(-a)^{n}\,_{2}F_{0}\left( -n,-x;-;\frac{-1}{a}\right) ,
\label{S2-ChHyper}
\end{equation}%
\begin{equation*}
C_{n}^{(a)}(x)=\sum_{k=0}^{n}\dbinom{n}{k}\dbinom{x}{k}k!(-a)^{n-k},
\end{equation*}%
or%
\begin{equation*}
C_{n}^{(a)}(x)=(-1)^{n}a^{-x}\Gamma (x+1)\Delta ^{n}\left[ \frac{a^{x}}{%
\Gamma (x-n+1)}\right] .
\end{equation*}%
Here, $_{r}F_{s}$ denotes the ordinary hypergeometric series defined by%
\begin{equation*}
_{r}F_{s}\left( a_{1},a_{2},\ldots ,a_{r};b_{1},b_{2},\ldots ,b_{s};x\right)
=\sum_{k=0}^{\infty }\frac{(a_{1})_{k}(a_{2})_{k}\cdots (a_{r})_{k}}{%
(b_{1})_{k}(b_{2})_{k}\cdots (b_{s})_{k}}\frac{x^{k}}{k!},
\end{equation*}%
\begin{equation*}
(a)_{0}:=1,\quad (a)_{k}:=a(a+1)(a+2)\cdots (a+k-1),\quad k=1,2,3,\ldots
\end{equation*}%
Next, we summarize some basic properties of Charlier orthogonal polynomials
to be used in the sequel.



\begin{proposition}
\label{[S2]-Proposition1}Let $\{C_{n}^{(a)}\}_{n\geq 0}$ be the classical
Charlier SMOP. The following statements hold.

\begin{enumerate}
\item Three term recurrence relation. For every $n\geq 0$, 
\begin{equation}
C_{n+1}^{(a)}(x)=\left( x-\beta _{n}\right) C_{n}^{(a)}(x)-\gamma
_{n}C_{n-1}^{(a)}(x),  \label{S2-C3TRR}
\end{equation}%
with initial conditions $C_{-1}^{(a)}(x)=0$, $C_{0}^{(a)}(x)=1$, and
coefficients $\beta _{n}=n+a$, $\gamma _{n}=na$.

\item Structure relation. For every $n\in \mathbb{N}$,%
\begin{equation}
x\nabla C_{n}^{(a)}(x)=nC_{n}^{(a)}(x)+naC_{n-1}^{(a)}(x).  \label{S2-StR}
\end{equation}

\item Norm. For every $n\in \mathbb{N}$,%
\begin{equation}
||C_{n}^{(a)}||^{2}=\int_{0}^{\infty }\left( C_{n}^{(a)}(x)\right) ^{2}d\psi
^{(a)}(x)=n!a^{n},  \label{S2-Norms}
\end{equation}%
and therefore%
\begin{equation}
\frac{||C_{n}^{(a)}||^{2}}{||C_{n-1}^{(a)}||^{2}}=\gamma _{n}=na.
\label{S3-cocNorms}
\end{equation}

\item Second order difference equations. For every $n\in \mathbb{N}$ (see 
\cite[Ch. VI.1]{Chi78}),%
\begin{equation}
a\Delta ^{2}C_{n}^{(a)}(x)-(x+1-a-n)\Delta C_{n}^{(a)}(x)+nC_{n}^{(a)}(x)=0,
\label{EqDiffChar1}
\end{equation}%
and also we have the hypergeometric type equation (see \cite[Ch. 4]%
{Alvarez03} and \cite[§ 2.1]{Nikiforov91}),%
\begin{equation}
x\Delta \nabla C_{n}^{(a)}(x)+(a-x)\Delta C_{n}^{(a)}(x)+nC_{n}^{(a)}(x)=0.
\label{EqDiffChar2}
\end{equation}

\item First order difference relation. For every $n\in \mathbb{N}$ (see \cite%
[Ch. VI.1]{Chi78}),%
\begin{equation}
\Delta C_{n}^{(a)}(x)=nC_{n-1}^{(a)}(x).  \label{S2-StructRel}
\end{equation}
\end{enumerate}
\end{proposition}



We denote the $n$-th reproducing kernel by%
\begin{equation*}
K_{n}(x,y)=\sum_{k=0}^{n}\frac{C_{k}^{(a)}(x)C_{k}^{(a)}(y)}{%
||C_{k}^{(a)}||^{2}}.
\end{equation*}%
Then, for all $n\in \mathbb{N}$,%
\begin{equation*}
K_{n}(x,y)=\frac{1}{||C_{n}^{(a)}||^{2}}\frac{%
C_{n+1}^{(a)}(x)C_{n}^{(a)}(y)-C_{n+1}^{(a)}(y)C_{n}^{(a)}(x)}{x-y}.
\end{equation*}%
Provided $\Delta ^{k}f(x)=\Delta \left[ \Delta ^{k-1}f(x)\right] $, for the
partial finite difference of $K_{n}(x,y)$ we will use the following notation%
\begin{equation}
K_{n}^{\left( i,j\right) }(x,y)=\Delta _{x}^{i}\left[ \Delta _{y}^{j}\left[
K_{n}\left( x,y\right) \right] \right] =\sum_{k=0}^{n}\frac{\Delta
^{i}C_{k}^{(a)}(x)\Delta ^{j}C_{k}^{(a)}(y)}{||C_{k}^{(a)}||^{2}}
\label{S2-S2-Kij}
\end{equation}%
and we observe the following consequence, provided that $c$ is not a zero of 
$C_{n}^{(a)}(x)$ for any $n$%
\begin{equation}
\frac{\lbrack \Delta C_{n}^{(a)}(c)]^{2}}{||C_{n}^{(a)}||^{2}}%
=K_{n}^{(1,1)}(c,c)-K_{n-1}^{(1,1)}(c,c)>0.  \label{S2-Nucleos}
\end{equation}

Finally, the following assumption will be needed throughout the paper. A
straightforward consequence of (\ref{S1-SobtypInnPr}), is that the
multiplication operator by $(x-c)(x-c-1)$ is symmetric with respect to such
a discrete Sobolev inner product. Indeed, for any $p,q\in \mathbb{P}$ we have%
\begin{eqnarray}
\left\langle (x-c)(x-c-1)p(x),q(x)\right\rangle _{\lambda } &=&\left\langle
p(x),(x-c)(x-c-1)q(x)\right\rangle _{\lambda }  \notag \\
&=&\left\langle (x-c)(x-c-1)p(x),q(x)\right\rangle  \label{S2-POk} \\
&=&\left\langle p(x),(x-c)(x-c-1)q(x)\right\rangle .  \notag
\end{eqnarray}



\section{Connection formulas and hypergeometric representation}

\label{[S3]-ConnForms}



In this Section, we modify the connection formula for the Sobolev-type
Charlier polynomials given in \cite[(2.13)]{B-JCAM95}, in order to obtain
alternative representations for $Q_{n}^{\lambda }(x)$ in terms of several
consecutive polynomials from the SMOP $\{C_{n}^{(a)}\}_{n\geq 0}$. In the
first of these new representations, we show that every coefficient can be
found in a very compact way, and directly related to the following parameters%
\begin{equation}
a_{n}=\frac{C_{n+1}^{(a)}(c)}{C_{n}^{(a)}(c)},\qquad b_{n}=\frac{1+\lambda
K_{n}^{(1,1)}(c,c)}{1+\lambda K_{n-1}^{(1,1)}(c,c)},\qquad n\geq 1.
\label{S3-Parameters}
\end{equation}%
From (\ref{S2-Nucleos}) we get%
\begin{equation*}
b_{n}=\frac{1+\lambda \left( K_{n-1}^{(1,1)}(c,c)+\frac{[\Delta
C_{n}^{(a)}(c)]^{2}}{||C_{n}^{(a)}||^{2}}\right) }{1+\lambda
K_{n-1}^{(1,1)}(c,c)}=1+\lambda \frac{\lbrack \Delta C_{n}^{(a)}(c)]^{2}}{%
||C_{n}^{(a)}||^{2}\left( 1+\lambda K_{n-1}^{(1,1)}(c,c)\right) }
\end{equation*}%
which will be always positive, because $\lambda $, $[\Delta
C_{n}^{(a)}(c)]^{2}$, $||C_{n}^{(a)}||^{2}$ are always positive and from (%
\ref{S2-S2-Kij}) we observe $K_{n-1}^{(1,1)}(c,c)>0$ as well. Notice that%
\begin{equation*}
b_{n}-1=\lambda \frac{\lbrack \Delta C_{n}^{(a)}(c)]^{2}}{%
||C_{n}^{(a)}||^{2}\left( 1+\lambda K_{n-1}^{(1,1)}(c,c)\right) }>0.
\end{equation*}

Having said that, we begin with the connection formula provided in \cite[%
(2.13)]{B-JCAM95}%
\begin{equation}
Q_{n}^{\lambda}(x)=A_{1}(x;n)C_{n}^{(a)}(x)+B_{1}(x;n)C_{n-1}^{(a)}(x),\quad
n\geq 1,  \label{S3-Bavinck(2.13)}
\end{equation}%
where the coefficients are the rational functions%
\begin{eqnarray*}
A_{1}(x;n) &=&1-\lambda \frac{\Delta Q_{n}^{\lambda }(c)}{%
||C_{n-1}^{(a)}||^{2}}\frac{C_{n-1}^{(a)}(c)+\left( x-c\right) \Delta
C_{n-1}^{(a)}(c)}{(x-c)(x-c-1)}, \\
B_{1}(x;n) &=&\lambda \frac{\Delta Q_{n}^{\lambda }(c)}{||C_{n-1}^{(a)}||^{2}%
}\frac{C_{n}^{(a)}(c)+\left( x-c\right) \Delta C_{n}^{(a)}(c)}{(x-c)(x-c-1)},
\end{eqnarray*}%
and $\Delta Q_{n}^{\lambda }(c)$ is given by (see \cite[(2.4)]{B-JCAM95})%
\begin{equation}
\Delta Q_{n}^{\lambda }(c)=\frac{\Delta C_{n}^{(a)}(c)}{1+\lambda
K_{n-1}^{(1,1)}(c,c)}.  \label{S3-Bavinck(2.4)}
\end{equation}%
By straightforward calculation, from (\ref{S3-Bavinck(2.13)}) we can write%
\begin{eqnarray}
(x-c)(x-c-1)A_{1}(x;n) &=&(x-c)(x-c-1)+A_{11}(n;c)(x-c)+A_{10}(n;c),
\label{Prop2eq1} \\
(x-c)(x-c-1)B_{1}(x;n) &=&B_{11}(n;c)(x-c)+B_{10}(n;c),  \label{Prop2eq2}
\end{eqnarray}%
where%
\begin{eqnarray*}
A_{11}(n;c) &=&-\lambda \frac{\Delta Q_{n}^{\lambda }(c)}{%
||C_{n-1}^{(a)}||^{2}}\Delta C_{n-1}^{(a)}(c),\quad A_{10}(n;c)=-\lambda 
\frac{\Delta Q_{n}^{\lambda }(c)}{||C_{n-1}^{(a)}||^{2}}C_{n-1}^{(a)}(c), \\
B_{11}(n;c) &=&\lambda \frac{\Delta Q_{n}^{\lambda }(c)}{%
||C_{n-1}^{(a)}||^{2}}\Delta C_{n}^{(a)}(c),\quad B_{10}(n;c)=\lambda \frac{%
\Delta Q_{n}^{\lambda }(c)}{||C_{n-1}^{(a)}||^{2}}C_{n}^{(a)}(c).
\end{eqnarray*}%
Thus%
\begin{equation*}
(x-c)(x-c-1)Q_{n}^{\lambda }(x)=
\end{equation*}%
\begin{eqnarray}
&&(x-c)(x-c-1)C_{n}^{(a)}(x)+A_{11}(n;c)(x-c)C_{n}^{(a)}(x)+A_{10}(n;c)C_{n}^{(a)}(x)
\label{S3-CFExp} \\
&&+B_{11}(n;c)(x-c)C_{n-1}^{(a)}(x)+B_{10}(n;c)C_{n-1}^{(a)}(x).  \notag
\end{eqnarray}

From (\ref{S3-Bavinck(2.4)}), and (\ref{S3-cocNorms}) we get%
\begin{eqnarray*}
A_{11}(n;c) &=&-\lambda \frac{\gamma _{n}}{||C_{n}^{(a)}||^{2}}\frac{\Delta
C_{n}^{(a)}(c)}{1+\lambda K_{n-1}^{(1,1)}(c,c)}\Delta C_{n-1}^{(a)}(c)\frac{%
\Delta C_{n}^{(a)}(c)}{\Delta C_{n}^{(a)}(c)} \\
&=&\frac{-\gamma _{n}}{1+\lambda K_{n-1}^{(1,1)}(c,c)}\frac{\Delta
C_{n-1}^{(a)}(c)}{\Delta C_{n}^{(a)}(c)}\frac{\lambda \lbrack \Delta
C_{n}^{(a)}(c)]^{2}}{||C_{n}^{(a)}||^{2}}.
\end{eqnarray*}%
Next, from (\ref{S2-StructRel}) and (\ref{S2-Nucleos}) we deduce%
\begin{equation*}
A_{11}(n;c)=-\gamma _{n-1}\frac{1}{\frac{C_{n-1}^{(a)}(c)}{C_{n-2}^{(a)}(c)}}%
\frac{\left( 1+\lambda K_{n}^{(1,1)}(c,c)\right) -\left( 1+\lambda
K_{n-1}^{(1,1)}(c,c)\right) }{1+\lambda K_{n-1}^{(1,1)}(c,c)},
\end{equation*}%
and taking into account (\ref{S3-Parameters}) and the coefficients in (\ref%
{S2-C3TRR}), we observe%
\begin{equation*}
A_{11}(n;c)=-\frac{\gamma _{n-1}}{a_{n-2}}\left( b_{n}-1\right) =-a\frac{n-1%
}{a_{n-2}}\left( b_{n}-1\right) .
\end{equation*}%
Concerning $A_{10}(n;c)$, combining (\ref{S3-Bavinck(2.4)}), (\ref%
{S3-cocNorms}), (\ref{S2-StructRel}) and (\ref{S2-Nucleos}) in the same way
yields%
\begin{eqnarray*}
A_{10}(n;c) &=&-\lambda \frac{1}{1+\lambda K_{n-1}^{(1,1)}(c,c)}\frac{%
[\Delta C_{n}^{(a)}(c)]^{2}}{\frac{||C_{n}^{(a)}||^{2}}{\gamma _{n}}}\frac{a%
}{na} \\
&=&-\frac{\left( 1+\lambda K_{n}^{(1,1)}(c,c)\right) -\left( 1+\lambda
K_{n-1}^{(1,1)}(c,c)\right) }{1+\lambda K_{n-1}^{(1,1)}(c,c)}a \\
&=&-a(b_{n}-1).
\end{eqnarray*}%
By proceeding with few more steps in the same fashion, we obtain%
\begin{equation*}
B_{11}(n;c)=an(b_{n}-1)\quad \text{and}\quad B_{10}(n;c)=a\,a_{n-1}(b_{n}-1).
\end{equation*}%
Thus, we have proved the following result



\begin{proposition}
\label{[S3]-Proposition2}For every $n\geq 1$, $\lambda \in \mathbb{R}_{+}$, $%
c\in \mathbb{R}$ such that $\psi ^{(a)}$ has no points of increase in the
interval $(c,c+1)$, and $a>0$, the four coefficients $A_{11}(n;c)$, $%
A_{10}(n;c)$, $B_{11}(n;c)$, and $B_{10}(n;c)$ can be expressed in terms of $%
a_{n}$, $b_{n}$ given in (\ref{S3-Parameters}) as follows%
\begin{equation*}
\begin{array}{ll}
A_{11}(n;c)=-a\frac{n-1}{a_{n-2}}\left( b_{n}-1\right) , & 
A_{10}(n;c)=-a(b_{n}-1),\smallskip \\ 
B_{11}(n;c)=an(b_{n}-1), & B_{10}(n;c)=a\,a_{n-1}(b_{n}-1).%
\end{array}%
\end{equation*}%
and therefore%
\begin{equation*}
(x-c)(x-c-1)Q_{n}^{\lambda }(x)=
\end{equation*}%
\begin{eqnarray}
&&(x-c)(x-c-1)C_{n}^{(a)}(x)-a\frac{n-1}{a_{n-2}}\left( b_{n}-1\right)
(x-c)C_{n}^{(a)}(x)-a(b_{n}-1)C_{n}^{(a)}(x)  \label{S3-NewCForm} \\
&&+an(b_{n}-1)(x-c)C_{n-1}^{(a)}(x)+a\,a_{n-1}(b_{n}-1)C_{n-1}^{(a)}(x). 
\notag
\end{eqnarray}
\end{proposition}

Concerning the norm of the polynomials $Q_{n}^{\lambda }(x)$ we can state
the following



\begin{theorem}
\label{[S3]-Theorem1}Let $\left\{ Q_{n}^{\lambda }\right\} _{n\geq 0}$ be
the sequence of Sobolev-type Charlier orthogonal polynomials defined by (\ref%
{S3-Bavinck(2.13)}). Then, for every $n\geq 1$, $\lambda \in \mathbb{R}_{+}$%
, $c\in \mathbb{R}$ such that $\psi ^{(a)}$ has no points of increase in the
interval $(c,c+1)$, and $a>0$, the norm of these polynomials, orthogonal
with respect to (\ref{S1-SobtypInnPr}) is%
\begin{equation}
||Q_{n}^{\lambda }||_{\lambda }^{2}=||C_{n}^{(a)}||^{2}+\gamma
_{n}(b_{n}-1)||C_{n-1}^{(a)}||^{2}.  \label{S3-NormSP}
\end{equation}
\end{theorem}



\begin{proof}
Clearly%
\begin{equation*}
||Q_{n}^{\lambda }||_{\lambda }^{2}=\langle Q_{n}^{\lambda
}(x),(x-c)(x-c-1)\pi _{n-2}(x)\rangle _{\lambda },
\end{equation*}%
for every monic polynomial $\pi _{n-2}$ of degree $n-2$. From (\ref{S2-POk})
we have%
\begin{equation*}
||Q_{n}^{\lambda }||_{\lambda }^{2}=\langle (x-c)(x-c-1)Q_{n}^{\lambda
}(x),\pi _{n-2}(x)\rangle .
\end{equation*}%
Taking into account (\ref{S3-NewCForm}) and again (\ref{S2-POk}), by
orthogonality we deduce%
\begin{eqnarray*}
||Q_{n}^{\lambda }||_{\lambda }^{2} &=&\langle
C_{n}^{(a)}(x),(x-c)(x-c-1)\pi _{n-2}(x)\rangle \\
&&+\gamma _{n}(b_{n}-1)\langle (x-c)C_{n-1}^{(a)}(x),\pi _{n-2}(x)\rangle \\
&=&||C_{n}^{(a)}||^{2}+\gamma _{n}(b_{n}-1)||C_{n-1}^{(a)}||^{2}.
\end{eqnarray*}%
This completes the proof.
\end{proof}


This can be used to derive the following result



\begin{corollary}
\label{[S3]-Corollary1}Under the assumptions of Theorem \ref{[S3]-Theorem1},
the following expression%
\begin{equation*}
\frac{||Q_{n}^{\lambda }||_{\lambda }^{2}}{||C_{n}^{(a)}||^{2}}=\frac{%
1+\lambda K_{n}^{(1,1)}(c,c)}{1+\lambda K_{n-1}^{(1,1)}(c,c)}
\end{equation*}%
holds.
\end{corollary}



\begin{remark}
Combining (\ref{S2-Norms}) with (\ref{S3-NormSP}), one can obtain the
following explicit expression for the norm of the Sobolev-type Charlier
polynomials%
\begin{equation*}
||Q_{n}^{\lambda }||_{\lambda }^{2}=n!a^{n}b_{n}
\end{equation*}%
This expression can be used to simplify and make more explicit other
expressions throughout the paper.
\end{remark}



In the framework of signal theory, the above gives the ratio of the energy
of polynomials $Q_{n}^{\lambda }(x)$ and $C_{n}^{(a)}(x)$ with respect to
the norms associated with their corresponding inner products. The proof
follows immediately from (\ref{S3-NormSP}) and (\ref{S3-cocNorms}).

It is known that connection formulas are the main tool to study the
analytical properties of new families of OPS, in terms of other families of
OPS with well-known analytical properties. With this in view, we next
present the Sobolev-type Charlier polynomials $Q_{n}^{\lambda }(x)$ in terms
of only five consecutive monic Charlier polynomials. Notice that in \cite[%
Lemma 2.2]{B-JCAM95} is already proven that such an expansion exist, but
there, the corresponding coefficients for the Charlier case are not
explicitly given.



\begin{proposition}
\label{[S3]-Proposition3}For every $n\geq 1$, $\lambda \in \mathbb{R}_{+}$, $%
c\in \mathbb{R}$ such that $\psi ^{(a)}$ has no points of increase in the
interval $(c,c+1)$, and $a>0$, the monic Sobolev-type Charlier orthogonal
polynomials $Q_{n}^{\lambda }(x)$ have the following representation in terms
of only five consecutive classical Charlier polynomials%
\begin{equation*}
(x-c)(x-c-1)Q_{n}^{\lambda }(x)=
\end{equation*}%
\begin{equation}
C_{n+2}^{(a)}(x)+\sigma _{n,1}C_{n+1}^{(a)}(x)+\sigma
_{n,0}C_{n}^{(a)}(x)+\sigma _{n,-1}C_{n-1}^{(a)}(x)+\sigma
_{n,-2}C_{n-2}^{(a)}(x)  \label{S3-LemExp}
\end{equation}%
where%
\begin{eqnarray*}
\sigma _{n,1} &=&2(n+a-c)+A_{11}(n;c), \\
\sigma _{n,0} &=&a^{2}+(c-n)(n-2a+c+1)+\left( n+a-c\right) \left[
2n+A_{11}(n;c)\right] \\
&&\qquad +A_{10}(n;c)+B_{11}(n;c), \\
\sigma _{n,-1} &=&A_{11}(n;c)na+\left( n-1+a-c\right) \left[ 2an+B_{11}(n;c)%
\right] +B_{10}(n;c), \\
\sigma _{n,-2} &=&a(n-1)\left[ na+B_{11}(n;c)\right] .
\end{eqnarray*}%
An alternative more explicit expression for these coefficients is%
\begin{eqnarray}
\sigma _{n,1} &=&2(n+a-c)-a\frac{n-1}{a_{n-2}}\left( b_{n}-1\right) ,  \notag
\\
\sigma _{n,0} &=&a^{2}+(c-n)(n-2a+c+1)+a(n-1)(b_{n}-1)  \notag \\
&&\qquad +\left( n+a-c\right) \left( 2n-a\frac{n-1}{a_{n-2}}\left(
b_{n}-1\right) \right) ,  \label{morexpl} \\
\sigma _{n,-1} &=&an\left( n+a-c-1\right) (b_{n}+1)+a(b_{n}-1)\left(
a_{n-1}-an\frac{n-1}{a_{n-2}}\right) ,  \notag \\
\sigma _{n,-2} &=&a^{2}n(n-1)b_{n}.  \notag
\end{eqnarray}
\end{proposition}



\begin{proof}
After some cumbersome, but doable computations, using (\ref{S2-C3TRR}) we
obtain%
\begin{eqnarray*}
(x-c)(x-c-1)C_{n}^{(a)}(x) &=&C_{n+2}^{(a)}(x)+2(n+a-c)C_{n+1}^{(a)}(x) \\
&&+\left[ (a-c)^{2}+c-n+4an-2cn+n^{2}\right] C_{n}^{(a)}(x) \\
&&+2an\left( a-c+n-1\right) C_{n-1}^{(a)}(x)+na^{2}(n-1)C_{n-2}^{(a)}(x),
\end{eqnarray*}%
\begin{eqnarray*}
A_{11}(n;c)(x-c)C_{n}^{(a)}(x)
&=&A_{11}(n;c)C_{n+1}^{(a)}(x)+A_{11}(n;c)\left( n+a-c\right) C_{n}^{(a)}(x)
\\
&&+A_{11}(n;c)naC_{n-1}^{(a)}(x),
\end{eqnarray*}%
\begin{eqnarray*}
B_{11}(n;c)(x-c)C_{n-1}^{(a)}(x) &=&B_{11}(n;c)C_{n}^{(a)}(x)+B_{11}(n;c) 
\left[ n-1+a-c\right] C_{n-1}^{(a)}(x) \\
&&+B_{11}(n;c)(n-1)aC_{n-2}^{(a)}(x).
\end{eqnarray*}%
Combining all the above expressions into (\ref{S3-CFExp}) we obtain the
desired coefficients in (\ref{S3-LemExp}). To obtain (\ref{morexpl}), it is
enough to consider Proposition \ref{[S3]-Proposition2}. This completes the
proof.
\end{proof}



In the remaining of this Section, we derive a representation of the
Sobolev-type Charlier polynomials $Q_{n}^{\lambda }(x)$ as hypergeometric
functions.



\begin{proposition}
\label{[S3]-Proposition4}For every $n\geq 1$, $\lambda \in \mathbb{R}_{+}$, $%
c\in \mathbb{R}$ such that $\psi ^{(a)}$ has no points of increase in the
interval $(c,c+1)$, and $a>0$, the monic Sobolev-type Charlier orthogonal
polynomials $Q_{n}^{\lambda }(x)$ have the following hypergeometric
representation%
\begin{equation}
Q_{n}^{\lambda }(x)=\left( -a\right) ^{n-1}\left[ B_{1}(x;n)-aA_{1}(x;n)%
\right] \,_{3}F_{1}\left( -n,-x,1-\phi _{n,a}(x);-\phi _{n,a}(x);\frac{-1}{a}%
\right)  \label{S3-CSPHR}
\end{equation}%
where%
\begin{equation*}
\phi _{n,a}(x)=n\left( 1-a\frac{A_{1}(x;n)}{B_{1}(x;n)}\right) .
\end{equation*}
\end{proposition}



\begin{proof}
Substituting (\ref{S2-ChHyper}) into (\ref{S3-Bavinck(2.13)}) yields%
\begin{eqnarray*}
Q_{n}^{\lambda }(x) &=&\left( -a\right) ^{n}A_{1}(x;n)\sum_{k=0}^{\infty
}\left( -n\right) _{k}\left( -x\right) _{k}\frac{\left( -a^{-1}\right) ^{k}}{%
k!} \\
&&+\left( -a\right) ^{n-1}B_{1}(x;n)\sum_{k=0}^{\infty }\left( 1-n\right)
_{k}\left( -x\right) _{k}\frac{\left( -a^{-1}\right) ^{k}}{k!}.
\end{eqnarray*}%
Note that although the sums above are up to infinity, these are terminating
hypergeometric series, because the Pochhammer symbol $(-n)_{k}$ becomes zero
if $k>n+1$. By a straightforward calculation, we have%
\begin{equation*}
Q_{n}^{\lambda }(x)=\left( -a\right) ^{n}\sum_{k=0}^{\infty }\left[
A_{1}(x;n)-\frac{(n-k)B_{1}(x;n)}{an}\right] (-n)_{k}\left( -x\right) _{k}%
\frac{\left( -a^{-1}\right) ^{k}}{k!}.
\end{equation*}%
Now we write the expression in square brackets as a rational function in the
variable $k$%
\begin{equation*}
\frac{B_{1}(x;n)}{an}\left( k-\phi _{n,a}(x)\right)
\end{equation*}%
with%
\begin{equation*}
\phi _{n,a}(x)=n\left( 1-a\frac{A_{1}(x;n)}{B_{1}(x;n)}\right) ,
\end{equation*}%
and therefore%
\begin{equation*}
Q_{n}^{\lambda }(x)=-\left( -a\right) ^{n-1}\frac{B_{1}(x;n)}{n}%
\sum_{k=0}^{\infty }\left[ k-\phi _{n,a}(x)\right] \left( -n\right)
_{k}\left( -x\right) _{k}\frac{\left( -a^{-1}\right) ^{k}}{k!}.
\end{equation*}%
Now, using some properties of the Pochhammer symbol, we can assert%
\begin{equation*}
k-\phi _{n,a}(x)=-\phi _{n,a}(x)\frac{\left( 1-\phi _{n,a}(x)\right) _{k}}{%
\left( -\phi _{n,a}(x)\right) _{k}}
\end{equation*}%
which gives%
\begin{equation*}
Q_{n}^{\lambda }(x)=\left( -a\right) ^{n-1}\left[ B_{1}(x;n)-aA_{1}(x;n)%
\right] \sum_{k=0}^{\infty }\frac{\left( -n\right) _{k}\left( -x\right)
_{k}\left( 1-\phi _{n,a}(x)\right) _{k}}{\left( -\phi _{n,a}(x)\right) _{k}}%
\frac{\left( -a^{-1}\right) ^{k}}{k!}.
\end{equation*}%
This completes the proof.
\end{proof}



\section{Ladder operators and second order linear difference equations}

\label{[S4]-Ladder}



Our next result concerns the \textit{ladder (creation and annihilation)
operators}, and the \textit{second order linear difference equation}
satisfied by the elements of the family $\{Q_{n}^{\lambda }\}_{n\geq 0}$. In
the literature, we can find two versions of the second order difference
equation satisfied by the classical Charlier polynomials. The first one (\ref%
{EqDiffChar1}) appears, for example in (see \cite[p. 171, eq. (1.8)]{Chi78}%
). The other one (\ref{EqDiffChar2}) is known as the second order difference
equation of hypergeometric type, and appears in \cite[Ch. 4]{Alvarez03} and 
\cite[§ 2.1]{Nikiforov91}. As the authors claimed in \cite[§ 2.1, p. 20]%
{Nikiforov91}, this equation \textquotedblleft ...arises also in some other
problems and has its own meaning\textquotedblright\ [\textit{sic}]. Thus as
it is proven, for example, in \cite[p. 101]{Alvarez03}, the polynomial
solutions of (\ref{EqDiffChar2}) satisfy that their $k$-th finite
differences $\Delta_{x}^{k}C_{n}^{(a)}(x)$ are in turn polynomial solutions
of a difference equation of the same kind.

In this Section, we will find the corresponding second order difference
equations of second order that $Q_{n}^{\lambda }(x)$ satisfy. We will also
obtain two different versions of this second order difference equation,
which generalize respectively (\ref{EqDiffChar1}) and (\ref{EqDiffChar2})
when $\lambda >0$.

We first provide the second order difference equation satisfied by $%
Q_{n}^{\lambda }(x)$ in the same way as in (\ref{EqDiffChar1}). The
technique used is based on the connection formula (\ref{S3-Bavinck(2.13)}),
the three term recurrence relation (\ref{S2-C3TRR}) satisfied by $%
\{C_{n}^{(a)}\}_{n\geq 0}$, and the simple difference relation (\ref%
{S2-StructRel}). We begin by proving some Lemmas that are needed to
substantiate our next results, namely Theorems \ref{[S4]-Theorem2} and \ref%
{[S4]-Theorem4}.



\begin{lemma}
\label{[S4]-Lemma1}Under the same hypothesis of Proposition \ref%
{[S3]-Proposition2}, for the SMOP $\{Q_{n}^{\lambda }\}_{n\geq 0}$ and $%
\{C_{n}^{(a)}\}_{n\geq 0}$ we have%
\begin{equation}
\Delta {Q_{n}^{\lambda }(x)}%
=C_{1}(x;n)C_{n}^{(a)}(x)+D_{1}(x;n)C_{n-1}^{(a)}(x),  \label{S4-DerQ-C1D1}
\end{equation}%
where%
\begin{equation}
\begin{array}{rcl}
C_{1}(x;n) & = & {\displaystyle A_{1}(x+1;n)-A_{1}(x;n)-\frac{1}{a}%
B_{1}(x+1;n)},\smallskip \\ 
D_{1}(x;n) & = & {\displaystyle nA_{1}(x+1;n)-B_{1}(x;n)+\frac{x-n+1}{a}%
B_{1}(x+1;n)}.%
\end{array}
\label{S4-C1D1}
\end{equation}%
with $A_{1}(x;n)$ and $B_{1}(x;n)$ given in (\ref{S3-Bavinck(2.13)}).
\end{lemma}



\begin{proof}
Shifting the index in (\ref{S2-StructRel}) as $n\rightarrow n-1$, and using (%
\ref{S2-C3TRR}) we obtain 
\begin{equation}
\Delta C_{n-1}^{(a)}(x)=\frac{-1}{a}C_{n}^{(a)}(x)+\frac{x-n-a+1}{a}%
C_{n-1}^{(a)}(x).  \label{S4-Ch1Diff}
\end{equation}%
Next, using the property (\ref{S2-RegProd}) we apply the forward difference
operator $\Delta $ in both sides of (\ref{S3-Bavinck(2.13)}), which yields%
\begin{eqnarray*}
\Delta {Q_{n}^{\lambda }(x)} &=&\Delta \left[ C_{n}^{(a)}(x)A_{1}(x;n)\right]
+\Delta \left[ C_{n-1}^{(a)}(x)B_{1}(x;n)\right] \\
&=&C_{n}^{(a)}(x)\Delta A_{1}(x;n)+\left[ \Delta A_{1}(x;n)+A_{1}(x;n)\right]
\Delta C_{n}^{(a)}(x) \\
&&+C_{n-1}^{(a)}(x)\Delta B_{1}(x;n)+\left[ \Delta B_{1}(x;n)+B_{1}(x;n)%
\right] \Delta C_{n-1}^{(a)}(x)
\end{eqnarray*}%
Substituting (\ref{S2-StructRel}) and (\ref{S4-Ch1Diff}) into the above
expression the Lemma follows.
\end{proof}



\begin{lemma}
\label{[S4]-Lemma2}The sequences of monic polynomials $\{Q_{n}^{\lambda
}\}_{n\geq 0}$ and $\{C_{n}^{(a)}\}_{n\geq 0}$ are also related by%
\begin{eqnarray}
{Q_{n-1}^{\lambda }(x)}
&=&A_{2}(x;n)C_{n}^{(a)}(x)+B_{2}(x;n)C_{n-1}^{(a)}(x),  \label{S4-Qnm1-A2D2}
\\
\Delta Q_{n-1}^{\lambda }(x)
&=&C_{2}(x;n)C_{n}^{(a)}(x)+D_{2}(x;n)C_{n-1}^{(a)}(x),
\label{S4-DxQnm1-C2D2}
\end{eqnarray}%
where%
\begin{eqnarray}
A_{2}(x;n) &=&\frac{-B_{1}(x;n-1)}{a(n-1)},  \notag \\
B_{2}(x;n) &=&A_{1}(x;n-1)+\frac{x-n-a+1}{a(n-1)}B_{1}(x;n-1),  \notag \\
C_{2}(x;n) &=&\frac{aB_{1}(x;n-1)-a(n-1)A_{1}(x+1;n-1)-(x-n+2)B_{1}(x+1;n-1)%
}{a^{2}(n-1)},  \label{S4-Coefs-ABCD2} \\
D_{2}(x;n) &=&\frac{a(n-1)(1-n+x)A_{1}(x+1;n-1)+a(-1+a+n-x)B_{1}(x;n-1)}{%
a^{2}(n-1)}  \notag \\
&&-A_{1}(x;n-1)+\frac{\left[ (-2+n-x)(-1+n-x)-a(1+x)\right] B_{1}(x+1;n-1)}{%
a^{2}(n-1)}.  \notag
\end{eqnarray}
\end{lemma}


\begin{proof}
The proof of (\ref{S4-Qnm1-A2D2}) and (\ref{S4-DxQnm1-C2D2}) is a
straightforward consequence of (\ref{S3-Bavinck(2.13)}), (\ref{S2-StructRel}%
), Lemma \ref{[S4]-Lemma1}, and the three term recurrence relation (\ref%
{S2-C3TRR}) for the SMOP $\{C_{n}^{(a)}\}_{n\geq 0}$.
\end{proof}



\begin{remark}
\label{[S4]-Remark1} Observe that the set of coefficients (\ref{S4-C1D1})
and (\ref{S4-Coefs-ABCD2}) can be given strictly in terms of the following
known quantities: the coefficients $A_{1}(x;n)$ and $B_{1}(x;n)$ in (\ref%
{S3-Bavinck(2.13)}), and the characteristic parameters ot the problem $a$, $%
c $, $\lambda $ and $n$.
\end{remark}



The following Lemma shows the converse relation of (\ref{S3-Bavinck(2.13)}%
)--(\ref{S4-Qnm1-A2D2}) for the polynomials $C_{n}^{(a)}(x)$ and $%
C_{n-1}^{(a)}(x)$. That is, we express these two consecutive polynomials of $%
\{C_{n}^{(a)}\}_{n\geq 0}$ in terms of only two consecutive polynomials of
the SMOP $\{Q_{n}^{\lambda }\}_{n\geq 0}$.



\begin{lemma}
\label{[S4]-Lemma3}For every $n\geq 1$, $\lambda \in \mathbb{R}_{+}$, $c\in 
\mathbb{R}$ such that $\psi ^{(a)}$ has no points of increase in the
interval $(c,c+1)$, and $a>0$, the following expressions hold%
\begin{eqnarray}
C_{n}^{(a)}(x) &=&\frac{B_{2}(x;n)}{\Lambda (x;n)}Q_{n}^{\lambda }(x)-\frac{%
B_{1}(x;n)}{\Lambda (x;n)}Q_{n-1}^{\lambda }(x),  \label{S4-InvR-Pn} \\
C_{n-1}^{(a)}(x) &=&\frac{-A_{2}(x;n)}{\Lambda (x;n)}Q_{n}^{\lambda }(x)+%
\frac{A_{1}(x;n)}{\Lambda (x;n)}Q_{n-1}^{\lambda }(x),  \label{S4-InvR-Pnm1}
\end{eqnarray}%
where $\Lambda (x;n)$ is the determinant%
\begin{equation}
\Lambda (x;n)=%
\begin{vmatrix}
A_{1}(x;n) & B_{1}(x;n) \\ 
A_{2}(x;n) & B_{2}(x;n)%
\end{vmatrix}
\label{S4-Lambdaxn}
\end{equation}
\end{lemma}


\begin{proof}
Note that (\ref{S3-Bavinck(2.13)})--(\ref{S4-Qnm1-A2D2}) can be interpreted
as a system of two linear equations with two polynomial unknowns, namely $%
C_{n}^{(a)}(x)$ and $C_{n-1}^{(a)}(x)$, hence from Cramer's rule the Lemma,
and the expression for $\Lambda (x;n)$ follow in a straightforward way.
\end{proof}



The proof of the next Theorem \ref{[S4]-Theorem2} easily follows from Lemmas %
\ref{[S4]-Lemma1}, \ref{[S4]-Lemma2} and \ref{[S4]-Lemma3}. Replacing (\ref%
{S4-InvR-Pn})--(\ref{S4-InvR-Pnm1}) in (\ref{S4-DerQ-C1D1}) and (\ref%
{S4-DxQnm1-C2D2}) one obtains the ladder difference equations%
\begin{eqnarray*}
\Delta {Q_{n}^{\lambda }(x)} &=&\left[ \frac{C_{1}(x;n)B_{2}(x;n)}{\Lambda
(x;n)}-\frac{D_{1}(x;n)A_{2}(x;n)}{\Lambda (x;n)}\right] Q_{n}^{\lambda }(x)
\\
&&+\left[ \frac{A_{1}(x;n)D_{1}(x;n)}{\Lambda (x;n)}-\frac{%
C_{1}(x;n)B_{1}(x;n)}{\Lambda (x;n)}\right] Q_{n-1}^{\lambda }(x),
\end{eqnarray*}%
\begin{eqnarray*}
\Delta {Q_{n-1}^{\lambda }(x)} &=&\left[ \frac{C_{2}(x;n)B_{2}(x;n)}{\Lambda
(x;n)}-\frac{A_{2}(x;n)D_{2}(x;n)}{\Lambda (x;n)}\right] Q_{n}^{\lambda }(x)
\\
&&+\left[ \frac{A_{1}(x;n)D_{2}(x;n)}{\Lambda (x;n)}-\frac{%
C_{2}(x;n)B_{1}(x;n)}{\Lambda (x;n)}\right] Q_{n-1}^{\lambda }(x).
\end{eqnarray*}%
Observe that defining the following determinants for $k=1,2$%
\begin{equation}
\Xi _{k}(x;n)=\frac{1}{\Lambda (x;n)}%
\begin{vmatrix}
C_{1}(x;n) & A_{k}(x;n) \\ 
D_{1}(x;n) & B_{k}(x;n)%
\end{vmatrix}%
,\quad \Theta _{k}(x;n)=\frac{1}{\Lambda (x;n)}%
\begin{vmatrix}
C_{2}(x;n) & A_{k}(x;n) \\ 
D_{2}(x;n) & B_{k}(x;n)%
\end{vmatrix}%
,  \label{S3-DifPSI1}
\end{equation}%
we can express the above ladder difference equations in the compact way%
\begin{eqnarray}
\Xi _{2}(x;n){Q_{n}^{\lambda }(x)}-\Delta {Q_{n}^{\lambda }(x)} &=&\Xi
_{1}(x;n)Q_{n-1}^{\lambda }(x),  \label{Ladder1Psi} \\
\Theta _{1}(x;n)Q_{n-1}^{\lambda }(x)+\Delta {Q_{n-1}^{\lambda }(x)}
&=&\Theta _{2}(x;n){Q_{n}^{\lambda }(x)}.  \label{Ladder2Psi}
\end{eqnarray}

Next, rearranging terms in the above two equations, we conclude the
following result, which is fully equivalent to (\ref{[S2]-LoweringEq})--(\ref%
{[S2]-RaisingEq}).



\begin{theorem}[ladder difference operators]
\label{[S4]-Theorem2}For every $n\geq 1$, $\lambda \in \mathbb{R}_{+}$, $%
c\in \mathbb{R}$ such that $\psi ^{(a)}$ has no points of increase in the
interval $(c,c+1)$, and $a>0$, let $\mathfrak{d}_{n}$\ and $\mathfrak{d}%
_{n}^{\dag }$\ be the difference operators%
\begin{eqnarray*}
\mathfrak{d}_{n} &=&\Xi _{2}(x;n)\mathrm{I}-\Delta , \\
\mathfrak{d}_{n}^{\dag } &=&\Theta _{1}(x;n)\mathrm{I}+\Delta ,
\end{eqnarray*}%
where $\mathrm{I}$ is the identity and $\Delta $ the forward difference
operators respectively. The difference operators $\mathfrak{d}_{n}$ and $%
\mathfrak{d}_{n}^{\dag }$ are respectively lowering and raising difference
operators associated to the Sobolev-type Charlier SMOP, satisfying%
\begin{eqnarray}
\mathfrak{d}_{n}[Q_{n}^{\lambda }(x)] &=&\Xi _{1}(x;n)Q_{n-1}^{\lambda }(x),
\label{[S2]-LoweringEq} \\
\mathfrak{d}_{n}^{\dag }[Q_{n-1}^{\lambda }(x)] &=&\Theta _{2}(x;n){%
Q_{n}^{\lambda }(x)},  \label{[S2]-RaisingEq}
\end{eqnarray}%
with $\Xi _{k}(x;n)$, $\Theta _{k}(x;n)$, $k=1,2$ given in (\ref{S3-DifPSI1}%
). These four coefficients can be given only in terms of the coefficients $%
A_{1}(x;n)$, $B_{1}(x;n)$ in (\ref{S3-Bavinck(2.13)}) and the parameters $a$%
, $\lambda $, $c$, and $n$ throughout the set of equations (\ref{S3-DifPSI1}%
), (\ref{S4-C1D1}), (\ref{S4-Coefs-ABCD2}).
\end{theorem}



For a deeper discussion on raising and lowering difference operators we
refer the reader to \cite[Ch. 3]{Ism05}. We next provide the second order
linear difference equation satisfied by the SMOP $\{{Q}_{n}^{\lambda
}\}_{n\geq 0}$.

Next, the proof of Theorem \ref{[S4]-Theorem4} comes directly from the
ladder operators provided in Theorem \ref{[S4]-Theorem2}. The usual
technique (see, for example \cite[Th. 3.2.3]{Ism05}) consists in applying
the raising operator to both sides of the equation satisfied by the lowering
operator, i.e. the expression%
\begin{equation}
\mathfrak{d}_{n}^{\dag }\left[ \frac{1}{\Xi _{1}(x;n)}\mathfrak{d}%
_{n}[Q_{n}^{\lambda }(x)]\right] =\mathfrak{d}_{n}^{\dag }\left[
Q_{n-1}^{\lambda }(x)\right] =\Theta _{2}(x;n){Q_{n}^{\lambda }(x)}
\label{S4-LadderEq}
\end{equation}%
is indeed a second order difference equation for $Q_{n}^{\lambda }(x)$.



\begin{theorem}[2nd order difference equation]
\label{[S4]-Theorem4}For every $n\geq 1$, $\lambda \in \mathbb{R}_{+}$, $%
c\in \mathbb{R}$ such that $\psi ^{(a)}$ has no points of increase in the
interval $(c,c+1)$, and $a>0$, the Sobolev-type Charlier SMOP $%
\{Q_{n}^{\lambda }\}_{n\geq 0}$\ satisfies the second order difference
equation%
\begin{equation}
\Delta ^{2}Q_{n}^{\lambda }(x)+\mathcal{R}(x;n)\Delta Q_{n}^{\lambda }(x)+%
\mathcal{S}(x;n)Q_{n}^{\lambda }(x)=0,  \label{S4-2ndODE}
\end{equation}%
with rational coefficients%
\begin{eqnarray}
\mathcal{R}(x;n) &=&\frac{\left[ \text{$\Theta _{1}$}(x;n)-1\right] \text{$%
\Delta \Xi _{1}$}(x;n)}{\text{$\Xi _{1}$}(x;n)}-\text{$\Delta \Xi _{2}$}%
(x;n)+\text{$\Theta _{1}$}(x;n)-\text{$\Xi _{2}$}(x;n),  \notag \\
\mathcal{S}(x;n) &=&\text{$\Theta _{2}$}(x;n)\left[ \text{$\Xi _{1}$}(x;n)+%
\text{$\Delta \Xi _{1}$}(x;n)\right] -\text{$\Delta \Xi _{2}(x;n)-\Theta
_{1} $}(x;n)\text{$\Xi _{2}$}(x;n)  \label{S4-RS} \\
&&\qquad -\frac{\text{$\Xi _{2}(x;n)\left[ \text{$\Theta _{1}$}(x;n)-1\right]
\Delta \Xi _{1}(x;n)$}}{\text{$\Xi _{1}$}(x;n)}.  \notag
\end{eqnarray}
\end{theorem}



\begin{proof}
From (\ref{S4-LadderEq}), using the definitions for $\mathfrak{d}_{n}^{\dag
} $, $\mathfrak{d}_{n}$ and the quotient rule (\ref{S2-QuoRule}) we obtain%
\begin{equation*}
\mathfrak{d}_{n}^{\dag }\left[ \frac{1}{\Xi _{1}(x;n)}\mathfrak{d}%
_{n}[Q_{n}^{\lambda }(x)]\right] -\Theta _{2}(x;n){Q_{n}^{\lambda }(x)}=
\end{equation*}%
\begin{eqnarray*}
&&\left( \Theta _{2}(x;n)-\frac{\Theta _{1}(x;n)\Xi _{2}(x;n)}{\Xi _{1}(x;n)}%
-\frac{\Xi _{1}(x;n)\Delta \Xi _{2}(x;n)-\Xi _{2}(x;n)\Delta \Xi _{1}(x;n)}{%
\Xi _{1}(x;n)\left( \Xi _{1}(x;n)+\Delta \Xi _{1}(x;n)\right) }\right)
Q_{n}^{\lambda }(x) \\
&&+\left( \frac{\Theta _{1}(x;n)}{\Xi _{1}(x;n)}-\frac{\Xi _{2}(x;n)}{\Xi
_{1}(x;n)}-\frac{\Xi _{1}(x;n)\Delta \Xi _{2}(x;n)-\Xi _{2}(x;n)\Delta \Xi
_{1}(x;n)+\Delta \Xi _{1}(x;n)}{\Xi _{1}(x;n)\left( \Xi _{1}(x;n)+\Delta \Xi
_{1}(x;n)\right) }\right) \Delta Q_{n}^{\lambda }(x) \\
&&+\left( \frac{1}{\Xi _{1}(x;n)}-\frac{\Delta \Xi _{1}(x;n)}{\Xi
_{1}(x;n)\left( \Xi _{1}(x;n)+\Delta \Xi _{1}(x;n)\right) }\right) \Delta
^{2}Q_{n}^{\lambda }(x)=0.
\end{eqnarray*}%
Multiplying all the equation by $\Xi _{1}(x;n)+\Delta \Xi _{1}(x;n)$ we have 
$1$ as the coefficient of $\Delta ^{2}Q_{n}^{\lambda }(x)$, so we finally
obtain (\ref{S4-RS}). This completes the proof.
\end{proof}



Doing few more computations, we can obtain the coefficients just in terms of
the functions $A_{1}(x;n)$ and $B_{1}(x;n)$ of the connection formula (\ref%
{S3-Bavinck(2.13)}), and the other parameters $a$, $\lambda $, $c$, and $n$.
Being%
\begin{equation*}
\mathcal{R}(x;n)=\frac{\mathcal{B}(x;n)}{\mathcal{A}(x;n)},\quad \text{\ and 
}\mathcal{S}(x;n)=\frac{\mathcal{C}(x;n)}{\mathcal{A}(x;n)},
\end{equation*}%
we have%
\begin{eqnarray*}
\mathcal{A}(x;n) &=&a^{2}A_{1}(x+1;n)\left[ nA_{1}(x;n)-B_{1}(x;n)\right] \\
&&+aB_{1}(x+1;n)\left[ (-n+x+1)A_{1}(x;n)+B_{1}(x;n)\right] ,
\end{eqnarray*}%
\begin{eqnarray*}
\mathcal{B}(x;n) &=&2a^{2}nA_{1}(x+1;n)A_{1}(x;n) \\
&&-(a-n+x+1)anA_{1}(x+2;n)A_{1}(x;n) \\
&&-\left( -n(a+2x+3)+n^{2}+(x+1)(x+2)\right) A_{1}(x;n)B_{1}(x+2;n) \\
&&+2a(-n+x+1)A_{1}(x;n)B_{1}(x+1;n) \\
&&-2a^{2}A_{1}(x+1;n)B_{1}(x;n)+a(a-n)A_{1}(x+2;n)B_{1}(x;n) \\
&&+2aB_{1}(x+1;n)B_{1}(x;n)-(a-n+x+2)B_{1}(x+2;n)B_{1}(x;n),
\end{eqnarray*}%
\begin{eqnarray*}
\mathcal{C}(x;n)
&=&n(a+2x+3)B_{1}(x+2;n)-n^{2}B_{1}(x+2;n)-(x+1)(x+2)A_{1}(x;n)B_{1}(x+2;n)
\\
&&-(a-n+x+2)B_{1}(x;n)B_{1}(x+2;n) \\
&&+(x+1)(-n+x+2)A_{1}(x+1;n)B_{1}(x+2;n) \\
&&+(x+1)B_{1}(x+2;n)B_{1}(x+1;n)+a(a-n)A_{1}(x+2;n)B_{1}(x;n) \\
&&+aA_{1}(x+1;n)\left[ n(x+1)A_{1}(x+2;n)-aB_{1}(x;n)\right] \\
&&+a^{2}nA_{1}(x+1;n)A_{1}(x;n)+na(-a+n-x-1)A_{1}(x+2;n)A_{1}(x;n) \\
&&+a(-n+x+1)A_{1}(x;n)B_{1}(x+1;n) \\
&&-aB_{1}(x+1;n)\left[ (x+1)B_{1}(x+2;n)-B_{1}(x;n)\right] ,
\end{eqnarray*}

We can use the above expressions to show that (\ref{S4-2ndODE}) becomes (\ref%
{EqDiffChar1}) when \ $\lambda =0$. In this case $Q_{n}^{\lambda
=0}(x)\equiv C_{n}^{(a)}(x)$, and therefore $A_{1}(x;n)=1$ and $B_{1}(x;n)=0$
in (\ref{S3-Bavinck(2.13)}). Under these assumptions, we get%
\begin{eqnarray*}
\mathcal{A}(x;n) &=&a^{2}n, \\
\mathcal{B}(x;n) &=&-an\left( x+1-a-n\right) , \\
\mathcal{C}(x;n) &=&an^{2}.
\end{eqnarray*}%
which allows us to recover (\ref{EqDiffChar1})\ from (\ref{S4-2ndODE}),
provided also that $a>0$ and $n\geq 1$.

Next, we study the generalization of the second order difference equation of
hypergeometric type (\ref{EqDiffChar1}). In order to prove our next Theorem,
we bring together few technical steps, presented in the following



\begin{lemma}
\label{[S4]-Lemma4}The monic Sobolev-type Charlier orthogonal polynomials $%
Q_{n}^{\lambda }(x)$ defined by (\ref{S3-CSPHR}) satisfy the following

\begin{enumerate}
\item[$i)$] 
\begin{equation}
\nabla Q_{n}^{\lambda
}(x)=C_{3}(x;n)C_{n}^{(a)}(x)+D_{3}(x;n)C_{n-1}^{(a)}(x),  \label{xDSP1}
\end{equation}%
where 
\begin{eqnarray*}
C_{3}(x;n) &=&\nabla A_{1}(x;n)+nx^{-1}A_{1}(x-1;n)-x^{-1}B_{1}(x-1;n), \\
D_{3}(x;n) &=&\nabla B_{1}(x;n)+nax^{-1}A_{1}(x-1;n)+\left( x-a\right)
x^{-1}B_{1}(x-1;n).
\end{eqnarray*}

\item[$ii)$] 
\begin{equation}
Q_{n-1}^{\lambda }\left( x\right) =\mathcal{F}_{1}(x;n)\nabla Q_{n}^{\lambda
}\left( x\right) +\mathcal{G}_{1}(x;n)Q_{n}^{\lambda }\left( x\right) ,
\label{QYN}
\end{equation}%
where 
\begin{eqnarray*}
\mathcal{F}_{1}(x;n) &=&-\frac{\Lambda (x;n)}{\Phi _{1}(x;n)},\quad \mathcal{%
G}_{1}(x;n)=\frac{\Phi _{2}(x;n)}{\Phi _{1}(x;n)}, \\
\Phi _{k}(x;n) &=&%
\begin{vmatrix}
C_{3}(x;n) & A_{k}(x;n) \\ 
D_{3}(x;n) & B_{k}(x;n)%
\end{vmatrix}%
,\quad k=1,2.
\end{eqnarray*}

\item[$iii)$] 
\begin{equation}
\Delta Q_{n}^{\lambda }\left( x\right) =\mathcal{F}_{2}(x;n)\nabla
Q_{n}^{\lambda }\left( x\right) +\mathcal{G}_{2}(x;n)Q_{n}^{\lambda }\left(
x\right) ,  \label{L4iii}
\end{equation}%
where 
\begin{eqnarray*}
\mathcal{F}_{2}(x;n) &=&-\Xi _{1}(x;n)\mathcal{F}_{1}(x;n), \\
\mathcal{G}_{2}(x;n) &=&\Xi _{2}(x;n)-\Xi _{1}(x;n)\mathcal{G}_{1}(x;n).
\end{eqnarray*}
\end{enumerate}
\end{lemma}



\begin{proof}

\begin{itemize}
\item[$i)$] Applying the forward operator $\nabla $ to (\ref%
{S3-Bavinck(2.13)}) and using \ref{S2-RegProd} we deduce%
\begin{eqnarray*}
\nabla Q_{n}^{\lambda }(x) &=&C_{n}^{(a)}(x)\nabla
A_{1}(x;n)+x^{-1}A_{1}(x-1;n)x\nabla C_{n}^{(a)}(x) \\
&&\qquad +C_{n-1}^{(a)}(x)\nabla B_{1}(x;n)+x^{-1}B_{1}(x-1;n)x\nabla
C_{n-1}^{(a)}(x).
\end{eqnarray*}%
Next, we use (\ref{S2-StR}) and the recurrence relation (\ref{S2-C3TRR}) to
expand $\nabla C_{n}^{(a)}(x)$ and $\nabla C_{n-1}^{(a)}(x)$ in terms of
only two consecutive Charlier polynomials $C_{n}^{(a)}(x)$ and $%
C_{n-1}^{(a)}(x)$. Pulling out common factors, we get (\ref{xDSP1}).

\item[$ii)$] Replacing (\ref{S4-InvR-Pn})-(\ref{S4-InvR-Pnm1}) in (\ref%
{xDSP1}), we obtain (\ref{QYN}).

\item[$iii)$] Combining (\ref{QYN}) with (\ref{Ladder1Psi}) we conclude (\ref%
{L4iii}) in a straightforward way.
\end{itemize}
\end{proof}



Now, we are ready to present the alternative version of the second order
difference equation (\ref{S4-2ndODE}) satisfied by the Sobolev-type Charlier
orthogonal polynomials.



\begin{theorem}
For every $n\geq 1$, $\lambda \in \mathbb{R}_{+}$, $c\in \mathbb{R}$ such
that $\psi ^{(a)}$ has no points of increase in the interval $(c,c+1)$, and $%
a>0$, the Sobolev-type Charlier SMOP $\{Q_{n}^{\lambda }\}_{n\geq 0}$\
satisfies the following second order difference equation%
\begin{equation}
\sigma \left( x\right) \Delta \nabla Q_{n}^{\lambda }\left( x\right) +\tau
\left( x\right) \Delta Q_{n}^{\lambda }\left( x\right) +\mu \left( x\right)
Q_{n}^{\lambda }\left( x\right) =0,  \label{EHT2}
\end{equation}%
where%
\begin{eqnarray}
\sigma \left( x\right) &=&\mathcal{F}_{2}(x;n),  \notag \\
\tau \left( x\right) &=&\Delta \mathcal{F}_{2}\left( x;n\right) +\mathcal{G}%
_{2}\left( x+1;n\right) +\mathcal{R}\left( x;n\right) ,  \label{EHT2coeff} \\
\mu \left( x\right) &=&\Delta \mathcal{G}_{2}\left( x;n\right) +\mathcal{S}%
\left( x;n\right) .  \notag
\end{eqnarray}%
Equation (\ref{EHT2}) becomes the hypergeometric type difference equation (%
\ref{EqDiffChar2}) when $\lambda =0.$
\end{theorem}



\begin{proof}
Replacing Lemma \ref{[S4]-Lemma4}-$iii)$ in (\ref{S4-2ndODE}), we obtain%
\begin{equation*}
\mathcal{F}_{2}\left( x+1;n\right) \Delta \nabla {Q_{n}^{\lambda }}\left(
x\right) +\nabla {Q_{n}^{\lambda }}\left( x\right) \Delta \mathcal{F}%
_{2}\left( x;n\right) +\left[ \mathcal{G}_{2}\left( x+1;n\right) +\mathcal{R}%
\left( x;n\right) \right] \Delta {Q_{n}^{\lambda }}\left( x\right)
\end{equation*}%
\begin{equation}
+\left[ \Delta \mathcal{G}_{2}\left( x;n\right) +\mathcal{S}\left(
x;n\right) \right] Q_{n}^{\lambda }(x)=0  \label{EHT1}
\end{equation}%
On the other hand, we have%
\begin{equation*}
\mathcal{F}_{2}\left( x+1;n\right) \Delta \nabla {Q_{n}^{\lambda }}\left(
x\right) +\nabla {Q_{n}^{\lambda }}\left( x\right) \Delta \mathcal{F}%
_{2}\left( x;n\right) =
\end{equation*}%
\begin{equation*}
\mathcal{F}_{2}\left( x;n\right) \Delta \nabla {Q_{n}^{\lambda }}\left(
x\right) +\Delta \mathcal{F}_{2}\left( x;n\right) \Delta \nabla {%
Q_{n}^{\lambda }}\left( x\right) +\nabla {Q_{n}^{\lambda }}\left( x\right)
\Delta \mathcal{F}_{2}\left( x;n\right) .
\end{equation*}%
Applying the property $\Delta \nabla =\Delta -\nabla $, we get%
\begin{equation*}
\mathcal{F}_{2}\left( x+1;n\right) \Delta \nabla {Q_{n}^{\lambda }}\left(
x\right) +\nabla {Q_{n}^{\lambda }}\left( x\right) \Delta \mathcal{F}%
_{2}\left( x;n\right) =\mathcal{F}_{2}\left( x;n\right) \Delta \nabla {%
Q_{n}^{\lambda }}\left( x\right) +\Delta \mathcal{F}_{2}\left( x;n\right)
\Delta {Q_{n}^{\lambda }}\left( x\right) .
\end{equation*}%
Thus, replacing the above expression into (\ref{EHT1}), we can assert%
\begin{equation*}
\mathcal{F}_{2}\left( x;n\right) \Delta \nabla {Q_{n}^{\lambda }}\left(
x\right) +\left[ \Delta \mathcal{F}_{2}\left( x;n\right) +\mathcal{G}%
_{2}\left( x+1;n\right) +\mathcal{R}\left( x;n\right) \right] \Delta {%
Q_{n}^{\lambda }}\left( x\right)
\end{equation*}%
\begin{equation*}
+\left[ \Delta \mathcal{G}_{2}\left( x;n\right) +\mathcal{S}\left(
x;n\right) \right] Q_{n}^{\lambda }(x)=0,
\end{equation*}%
which is (\ref{EHT2}).

Next, we evaluate (\ref{EHT2}) when $\lambda =0$. In this particular case,
we have%
\begin{eqnarray*}
A_{1}(x;n) &=&1,\,B_{1}(x;n)=0,\,C_{1}(x;n)=0,\,D_{1}(x;n)=n, \\
A_{2}(x;n) &=&0,\,B_{2}(x;n)=1,\,C_{2}(x;n)=-1/a,\,D_{2}(x;n)=\left(
x+1-n-a\right) /a, \\
C_{3}(x;n) &=&n/x,\,D_{3}(x;n)=na/x,\,\Lambda (x;n)=1, \\
\mathcal{R}(x;n) &=&\left( -x-1+a+n\right) /a,\,\mathcal{S}(x;n)=n/a, \\
\Xi _{1}(x;n) &=&-n,\,\Xi _{2}(x;n)=0,\,\Phi _{1}=-na/x,\,\Phi _{2}=n/x \\
\mathcal{F}_{1}(x;n) &=&x/na,\,\mathcal{G}_{1}(x;n)=-1/a,\,\mathcal{F}%
_{2}(x;n)=x/a,\,\mathcal{G}_{2}(x;n)=-n/a.
\end{eqnarray*}
Replacing all these values in (\ref{EHT2coeff}), yields%
\begin{eqnarray*}
\sigma \left( x\right) &=&\frac{x}{a}, \\
\tau \left( x\right) &=&\frac{\left( x+1\right) }{a}-\left( \frac{x}{a}%
\right) +\frac{-n}{a}+\frac{-x-1+a+n}{a}=1-\frac{1}{a}x, \\
\mu \left( x\right) &=&\frac{n}{a},
\end{eqnarray*}%
which leads to equation (\ref{EqDiffChar2}) divided by $a$. This completes
the proof.
\end{proof}



It is worth noting that equation (\ref{EHT2}) generalizes the equation of
hypergeometric type (\ref{EqDiffChar2}) for $\lambda >0$, but it is not
itself an equation of hypergeometric type. To be of hypergeometric type, (%
\ref{EHT2}) should fulfill important properties, such as those commented at
the beginning of this Section concerning the $k$-th finite differences of
its polynomial solutions $\Delta _{x}^{k}Q_{n}^{\lambda }(x)$. It can be
verified that, for example the first difference of its polynomials
solutions, do not satisfy an equation of the same type as (\ref{EHT2}).



\section{Recurrence formulas}

\label{[S5]-5TRR}



Here we present two recurrence formulas for the Sobolev-type Charlier
polynomials of the SMOP $\{Q_{n}^{\lambda }\}_{n\geq 0}$. The first one is a
five term recurrence relation, whose existence was proved in \cite[Prop. 3.1]%
{B-AA95} in a more general framework. Here we provide the explicit
coefficients in the Charlier case. The second one is a three term recurrence
relation whose coefficients are rational functions.

Next we find the explicit coefficients of the five term recurrence relation
satisfied by $Q_{n}^{\lambda }(x)$, whose existence is proven in \cite[%
Proposition 3.1, p. 236]{B-AA95}.



\begin{theorem}[Five term recurrence relation]
\label{[S5]-Theorem6}For every $n\geq 1$, $\lambda \in \mathbb{R}_{+}$, $%
c\in \mathbb{R}$ such that $\psi ^{(a)}$ has no points of increase in the
interval $(c,c+1)$, and $a>0$, the monic Sobolev-type Charlier polynomials $%
\{Q_{n}^{\lambda }\}_{n\geq 0}$, orthogonal with respect to (\ref%
{S1-SobtypInnPr}) satisfy the following five term recurrence relation%
\begin{equation*}
(x-c)(x-c-1)Q_{n}^{\lambda }(x)=
\end{equation*}%
\begin{equation*}
Q_{n+2}^{\lambda }(x)+\rho _{n,n+1}Q_{n+1}^{\lambda }(x)+\rho
_{n,n}Q_{n}^{\lambda }(x)+\rho _{n,n-1}Q_{n-1}^{\lambda }(x)+\rho
_{n,n-2}Q_{n-2}^{\lambda }(x),
\end{equation*}%
where%
\begin{equation*}
\rho _{n,n+1}=\frac{\sigma _{n+1,-1}||C_{n}^{(a)}||^{2}}{||Q_{n+1}^{\lambda
}||_{\lambda }^{2}}-\frac{\gamma _{n-1}}{a_{n-2}}\left( b_{n}-1\right) ,
\end{equation*}%
\begin{equation*}
\rho _{n,n}=\sigma _{n,0}\frac{||C_{n}^{(a)}||^{2}}{||Q_{n}^{\lambda
}||_{\lambda }^{2}}-(b_{n}-1)\left[ \frac{\gamma _{n-1}}{a_{n-2}}\left(
\sigma _{n,-1}\frac{||C_{n-1}^{(a)}||^{2}}{||Q_{n}^{\lambda }||_{\lambda
}^{2}}-\left( n+a-c-2\right) \right) +(a-\gamma _{n})\right] ,
\end{equation*}%
\begin{equation*}
\rho _{n,n-1}=\sigma _{n,-1}\frac{||C_{n-1}^{(a)}||^{2}}{||Q_{n-1}^{\lambda
}||_{\lambda }^{2}}-\frac{\gamma _{n-1}}{a_{n-2}}\left( b_{n}-1\right) \frac{%
||Q_{n}^{\lambda }||_{\lambda }^{2}}{||Q_{n-1}^{\lambda }||_{\lambda }^{2}},
\end{equation*}%
\begin{equation*}
\rho _{n,n-2}=\sigma _{n,-2}\frac{||C_{n-2}^{(a)}||^{2}}{||Q_{n-2}^{\lambda
}||_{\lambda }^{2}}.
\end{equation*}
\end{theorem}



\begin{proof}
We first consider the Fourier expansion of $(x-c)(x-c-1)Q_{n}^{\lambda }(x)$
in terms of the elements of $\{Q_{n}^{\lambda }\}_{n\geq 0}$%
\begin{equation*}
(x-c)(x-c-1)Q_{n}^{\lambda }(x)=Q_{n+2}^{\lambda }(x)+\sum_{k=0}^{n+1}\rho
_{n,k}Q_{k}^{\lambda }(x),
\end{equation*}%
where%
\begin{equation*}
\rho _{n,k}=\frac{\langle (x-c)(x-c-1)Q_{n}^{\lambda }(x),Q_{k}^{\lambda
}(x)\rangle _{\lambda }}{||Q_{k}^{\lambda }||_{\lambda }^{2}},\quad
k=0,\ldots ,n+1,
\end{equation*}%
and $\rho _{n,k}=0$ for $k=0,\ldots ,n-3$. Using (\ref{S3-NewCForm}) we
deduce%
\begin{eqnarray*}
\rho _{n,n+1} &=&\frac{\langle (x-c)(x-c-1)Q_{n}^{\lambda
}(x),Q_{n+1}^{\lambda }(x)\rangle _{\lambda }}{||Q_{n+1}^{\lambda
}||_{\lambda }^{2}} \\
&=&\frac{1}{||Q_{n+1}^{\lambda }||_{\lambda }^{2}}\langle
(x-c)(x-c-1)C_{n}^{(a)}(x),Q_{n+1}^{\lambda }(x)\rangle _{\lambda }-\frac{%
\gamma _{n-1}}{a_{n-2}}\left( b_{n}-1\right) .
\end{eqnarray*}%
Thus, from (\ref{S2-POk}) we get%
\begin{equation*}
\langle (x-c)(x-c-1)C_{n}^{(a)}(x),Q_{n+1}^{\lambda }(x)\rangle _{\lambda
}=\langle C_{n}^{(a)}(x),(x-c)(x-c-1)Q_{n+1}^{\lambda }(x)\rangle .
\end{equation*}%
Then, having into account (\ref{S3-LemExp}), we have%
\begin{equation*}
\langle (x-c)(x-c-1)C_{n}^{(a)}(x),Q_{n+1}^{\lambda }(x)\rangle =\sigma
_{n+1,-1}||C_{n}^{(a)}||^{2}
\end{equation*}%
and, in consequence%
\begin{equation*}
\rho _{n,n+1}=\frac{\sigma _{n+1,-1}||C_{n}^{(a)}||^{2}}{||Q_{n+1}^{\lambda
}||_{\lambda }^{2}}-\frac{\gamma _{n-1}}{a_{n-2}}\left( b_{n}-1\right) .
\end{equation*}%
In order to compute $\rho _{n,n}$, we use (\ref{S3-NewCForm})%
\begin{eqnarray*}
\rho _{n,n} &=&\frac{\left\langle (x-c)(x-c-1)Q_{n}^{\lambda
}(x),Q_{n}^{\lambda }(x)\right\rangle _{\lambda }}{||Q_{n}^{\lambda
}||_{\lambda }^{2}}=\frac{\langle (x-c)(x-c-1)C_{n}^{(a)}(x),Q_{n}^{\lambda
}(x)\rangle }{||Q_{n}^{\lambda }||_{\lambda }^{2}} \\
&&\qquad -\frac{\gamma _{n-1}}{a_{n-2}}\left( b_{n}-1\right) \frac{\langle
\left( x-c\right) C_{n}^{(a)}(x),Q_{n}^{\lambda }(x)\rangle _{\lambda }}{%
||Q_{n}^{\lambda }||_{\lambda }^{2}}+a(n-1)(b_{n}-1).
\end{eqnarray*}%
From (\ref{S2-C3TRR}), (\ref{S2-POk}) and (\ref{S3-LemExp}) we deduce%
\begin{equation*}
\frac{\langle (x-c)(x-c-1)C_{n}^{(a)}(x),Q_{n}^{\lambda }(x)\rangle }{%
||Q_{n}^{\lambda }||_{\lambda }^{2}}=\frac{\langle
C_{n}^{(a)}(x),(x-c)(x-c-1)Q_{n}^{\lambda }(x)\rangle }{||Q_{n}^{\lambda
}||_{\lambda }^{2}}=\sigma _{n,0}\frac{||C_{n}^{(a)}||^{2}}{||Q_{n}^{\lambda
}||_{\lambda }^{2}},
\end{equation*}%
and%
\begin{eqnarray*}
\frac{\langle \left( x-c\right) C_{n}^{(a)}(x),Q_{n}^{\lambda }(x)\rangle
_{\lambda }}{||Q_{n}^{\lambda }||_{\lambda }^{2}} &=&\frac{\langle
C_{n-1}^{(a)}(x),(x-c)(x-c-1)Q_{n}^{\lambda }(x)\rangle }{||Q_{n}^{\lambda
}||_{\lambda }^{2}}-\left( n+a-c-2\right) \\
&=&\sigma _{n,-1}\frac{||C_{n-1}^{(a)}||^{2}}{||Q_{n}^{\lambda }||_{\lambda
}^{2}}-\left( n+a-c-2\right) .
\end{eqnarray*}%
Hence%
\begin{eqnarray*}
\rho _{n,n} &=&\sigma _{n,0}\frac{||C_{n}^{(a)}||^{2}}{||Q_{n}^{\lambda
}||_{\lambda }^{2}}-\frac{\gamma _{n-1}\sigma _{n,-1}}{a_{n-2}}\left(
b_{n}-1\right) \frac{||C_{n-1}^{(a)}||^{2}}{||Q_{n}^{\lambda }||_{\lambda
}^{2}} \\
&&+\frac{\gamma _{n-1}}{a_{n-2}}\left( b_{n}-1\right) \left( n+a-c-2\right)
+\gamma _{n}(b_{n}-1)-a(b_{n}-1) \\
&=&\sigma _{n,0}\frac{||C_{n}^{(a)}||^{2}}{||Q_{n}^{\lambda }||_{\lambda
}^{2}}-(b_{n}-1)\left[ \frac{\gamma _{n-1}}{a_{n-2}}\left( \sigma _{n,-1}%
\frac{||C_{n-1}^{(a)}||^{2}}{||Q_{n}^{\lambda }||_{\lambda }^{2}}-\left(
n+a-c-2\right) \right) +(a-\gamma _{n})\right]
\end{eqnarray*}%
Similarly, we compute%
\begin{eqnarray*}
\rho _{n,n-1} &=&\frac{\langle (x-c)(x-c-1)Q_{n}^{\lambda
}(x),Q_{n-1}^{\lambda }(x)\rangle _{\lambda }}{||Q_{n-1}^{\lambda
}||_{\lambda }^{2}}=\frac{\langle Q_{n}^{\lambda
}(x),(x-c)(x-c-1)Q_{n-1}^{\lambda }(x)\rangle _{\lambda }}{%
||Q_{n-1}^{\lambda }||_{\lambda }^{2}} \\
&=&\frac{\langle Q_{n}^{\lambda }(x),(x-c)(x-c-1)C_{n-1}^{(a)}(x)\rangle
_{\lambda }}{||Q_{n-1}^{\lambda }||_{\lambda }^{2}}-\frac{\gamma _{n-1}}{%
a_{n-2}}\left( b_{n}-1\right) \frac{||Q_{n}^{\lambda }||_{\lambda }^{2}}{%
||Q_{n-1}^{\lambda }||_{\lambda }^{2}},
\end{eqnarray*}%
where%
\begin{eqnarray*}
\frac{\langle Q_{n}^{\lambda }(x),(x-c)(x-c-1)C_{n-1}^{(a)}(x)\rangle
_{\lambda }}{||Q_{n-1}^{\lambda }||_{\lambda }^{2}} &=&\frac{\langle
(x-c)(x-c-1)Q_{n}^{\lambda }(x),C_{n-1}^{(a)}(x)\rangle }{||Q_{n-1}^{\lambda
}||_{\lambda }^{2}} \\
&=&\sigma _{n,-1}\frac{||C_{n-1}^{(a)}||^{2}}{||Q_{n-1}^{\lambda
}||_{\lambda }^{2}}.
\end{eqnarray*}%
Thus%
\begin{equation*}
\rho _{n,n-1}=\sigma _{n,-1}\frac{||C_{n-1}^{(a)}||^{2}}{||Q_{n-1}^{\lambda
}||_{\lambda }^{2}}-\frac{\gamma _{n-1}}{a_{n-2}}\left( b_{n}-1\right) \frac{%
||Q_{n}^{\lambda }||_{\lambda }^{2}}{||Q_{n-1}^{\lambda }||_{\lambda }^{2}}.
\end{equation*}%
Finally, using the same arguments, we conclude%
\begin{eqnarray*}
\rho _{n,n-2} &=&\frac{\langle (x-c)(x-c-1)Q_{n}^{\lambda
}(x),Q_{n-2}^{\lambda }(x)\rangle _{\lambda }}{||Q_{n-2}^{\lambda
}||_{\lambda }^{2}} \\
&=&\frac{\langle Q_{n}^{\lambda }(x),(x-c)(x-c-1)C_{n-2}^{(a)}(x)\rangle
_{\lambda }}{||Q_{n-2}^{\lambda }||_{\lambda }^{2}} \\
&=&\frac{\langle (x-c)(x-c-1)Q_{n}^{\lambda }(x),C_{n-2}^{(a)}(x)\rangle }{%
||Q_{n-2}^{\lambda }||_{\lambda }^{2}} \\
&=&\sigma _{n,-2}\frac{||C_{n-2}^{(a)}||^{2}}{||Q_{n-2}^{\lambda
}||_{\lambda }^{2}}.
\end{eqnarray*}%
This completes the proof.
\end{proof}



It is worth emphasizing that the values of the coefficients in the above
five-term recurrence relation depend on a number of quantities involved in
the present problem, such the coefficients $\beta _{n}$, $\gamma _{n}$ in (%
\ref{S2-C3TRR}), parameters $a_{n}$, $b_{n}$ in (\ref{S3-Parameters}),
coefficients $\sigma _{n,1}$ to $\sigma _{n,-2}$ in (\ref{[S3]-Proposition3}%
), etc.

Next, we prove an alternative recurrence formula with rational coefficients,
that can be used to find the next polynomial ${Q_{n+1}^{\lambda }(x)}$, from
the two former consecutive polynomials ${Q_{n}^{\lambda }(x)}$ and ${%
Q_{n-1}^{\lambda }(x)}$ in the Sobolev-type Charlier SMOP $\{Q_{n}^{\lambda
}\}_{n\geq 0}$. Using the ladder difference equations (\ref{Ladder1Psi}) and
(\ref{Ladder2Psi}) we have%
\begin{eqnarray*}
\Xi _{2}(x;n){Q_{n}^{\lambda }(x)}-\Delta {Q_{n}^{\lambda }(x)} &=&\Xi
_{1}(x;n)Q_{n-1}^{\lambda }(x), \\
\Theta _{1}(x;n+1)Q_{n}^{\lambda }(x)+\Delta {Q_{n}^{\lambda }(x)} &=&\Theta
_{2}(x;n+1){Q_{n+1}^{\lambda }(x)}.
\end{eqnarray*}%
Simply adding the above two equations, for every $n\geq 1$ we obtain the
aforesaid three term recurrence relation%
\begin{equation}
{Q_{n+1}^{\lambda }(x)}=\tilde{\beta}(x;n)Q_{n}^{\lambda }(x)+\tilde{\gamma}%
(x;n)Q_{n-1}^{\lambda }(x),  \label{TTRRratcoeff}
\end{equation}%
with rational coefficients%
\begin{equation*}
\tilde{\beta}(x;n)=\frac{\Xi _{2}(x;n)+\Theta _{1}(x;n+1)}{\Theta _{2}(x;n+1)%
},\quad \text{and\quad }\tilde{\gamma}(x;n)=-\frac{\Xi _{1}(x;n)}{\Theta
_{2}(x;n+1)}.
\end{equation*}%
Having in mind Remark \ref{[S4]-Remark1}, we can express $\tilde{\beta}(x;n)$
and $\tilde{\gamma}(x;n)$ exclusively in terms of $A_{1}(x;n)$ and $%
B_{1}(x;n)$ in (\ref{S3-Bavinck(2.13)}), and the rest of involved parameters 
$a$, $\lambda $, $c$, and $n$. From (\ref{S3-DifPSI1}), (\ref{S4-Coefs-ABCD2}%
) and (\ref{S4-Lambdaxn}) we conclude%
\begin{equation*}
\tilde{\beta}(x;n)=\frac{a(n-1)A_{1}(x;n-1)\left[
(-a-n+x)A_{1}(x;n+1)+B_{1}(x;n+1)\right] }{B_{1}(x;n-1)\left[
(-a-n+x+1)A_{1}(x;n)+B_{1}(x;n)\right] +a(n-1)A_{1}(x;n-1)A_{1}(x;n)}+
\end{equation*}%
\begin{equation*}
\frac{B_{1}(x;n-1)\left[ a^{2}+a(n-2x-1)+(n-x)^{2}-n+x\right]
A_{1}(x;n+1)-(a+n-x-1)B_{1}(x;n+1)}{B_{1}(x;n-1)\left[
(-a-n+x+1)A_{1}(x;n)+B_{1}(x;n)\right] +a(n-1)A_{1}(x;n-1)A_{1}(x;n)}
\end{equation*}%
and%
\begin{equation*}
\tilde{\gamma}(x;n)=\frac{a(n-1)\left\{ (A_{1}(x;n+1)\left[
(a+n-x)B_{1}(x;n)-anA_{1}(x;n)\right] -B_{1}(x;n)B_{1}(x;n+1)\right\} }{%
B_{1}(x;n-1)\left[ (-a-n+x+1)A_{1}(x;n)+B_{1}(x;n)\right]
+a(n-1)A_{1}(x;n-1)A_{1}(x;n)}.
\end{equation*}%
As might be expected, when $\lambda =0$, we have $A_{1}(x;n)=1$ and $%
B_{1}(x;n)=0$ above, and therefore we recover the corresponding coefficients 
$\beta _{n}$ and $\gamma_{n}$ in (\ref{S2-C3TRR}). Thus,%
\begin{eqnarray*}
\tilde{\beta}(x;n) &=&\frac{a(n-1)(-a-n+x)}{a(n-1)}=x-n-a=x-\beta _{n}, \\
\tilde{\gamma}(x;n) &=&\frac{-a^{2}n(n-1)}{a(n-1)}=-an=-\gamma _{n}.
\end{eqnarray*}



\section{Asymptotic behavior of the zeros}

\label{[S6]-Zeros}



In this Section we obtain new results on zero behavior of the Sobolev-type
Charlier orthogonal polynomials $Q_{n}^{\lambda }(x)$, which are different
and complementary to that encountered in the literature so far. We will
analyze the behavior of zeros of $Q_{n}^{\lambda }(x)$ as a function of the
mass $\lambda $, when $\lambda $ tends from zero to infinity as well as we
characterize the exact values of $\lambda $ such the smallest (respectively,
the largest) zero of $\{{Q_{n}^{\lambda }}\}_{n\geq 0}$ is located outside
of $I=\mathrm{supp}(\psi ^{(a)})$. In order to do that, we use a technique
developed and proved in \cite[Lemma 1]{BDR-JCAM02} and \cite[Lemmas 1 and 2]%
{DMR-ANM10}, concerning the behavior and the asymptotics of the zeros of
linear combinations of two polynomials $h,g\in \mathbb{P}$ with interlacing
zeros, such that $f(x)=h_{n}(x)+\lambda g_{n}(x)$. From now on, we will
refer to this technique as the \textit{Interlacing Lemma}, and for the
convenience of the reader, we include here the part in which we are
interested.



\begin{lemma}[Interlacing Lemma]
\label{[S6]-InterlacingLemma8} Let $h_{n}(x)=a(x-x_{1})\cdots (x-x_{n})$ and 
$g_{n}(x)=b(x-y_{1})\cdots (x-y_{n})$ be polynomials with real and simple
zeros, where $a$ and $b$ are real positive constants. If%
\begin{equation*}
y_{1}<x_{1}<\cdots <y_{n}<x_{n},
\end{equation*}%
then, for any real constant $\lambda >0$, the polynomial%
\begin{equation*}
f(x)=h_{n}(x)+\lambda g_{n}(x)
\end{equation*}%
has $n$ real zeros $\eta _{1}<\cdots <\eta _{n}$ which interlace with the
zeros of $h_{n}(x)$ and $g_{n}(x)$ in the following way 
\begin{equation*}
y_{1}<\eta _{1}<x_{1}<\cdots <y_{n}<\eta _{n}<x_{n}.
\end{equation*}%
Moreover, each $\eta _{k}=\eta _{k}(\lambda )$ is a decreasing function of $%
\lambda $ and, for each $k=1,\ldots ,n$, 
\begin{equation*}
\lim_{\lambda \rightarrow \infty }\eta _{k}=y_{k}\quad \text{and}\quad
\lim_{\lambda \rightarrow \infty }\lambda \lbrack \eta _{k}-y_{k}]=\dfrac{%
-h_{n}(y_{k})}{g_{n}^{\prime }(y_{k})}.
\end{equation*}
\end{lemma}



Concerning the classical Charlier polynomials $C_{n}^{(a)}(x)$, we denote by 
$\{x_{n,r}\}_{r=1}^{n}$ their zeros, all arranged in an increasing order.
When $a>0$\ we have $\psi ^{(a)}$ is positive-definite so all the zeros $%
\{x_{n,r}\}_{r=1}^{n}$ are simple and lie in $(0,+\infty )$. At the same
time, it is known that $Q_{n}^{\lambda }(x)$ can have complex zeros, and
this fact depends entirely on the choice of the parameter $c$. For example,
in Figure \ref{S6-figura1} are shown the two complex zeros of $%
Q_{4}^{100}(x) $ (for $a=0.34$) mentioned in the last Remark of (\cite[p.27]%
{B-JCAM95}), whose numerical values are $0.00403781$, $1.12129$, and $%
2.74947\pm 0.403581\,i$.

\begin{figure}[th]
\centerline{\includegraphics[width=11cm,keepaspectratio]{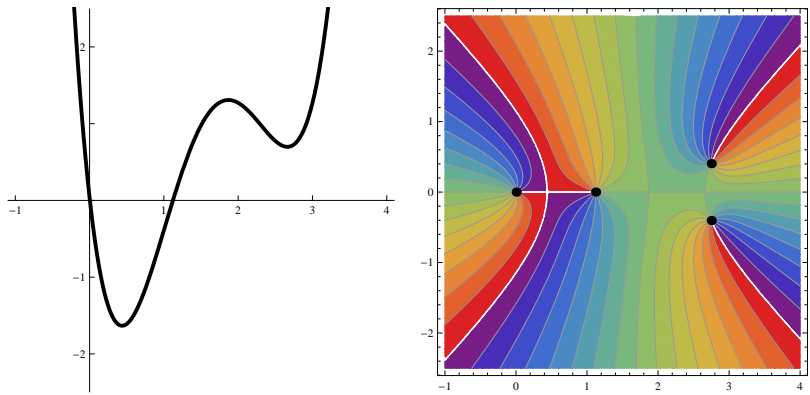}}
\caption{Graphs of $Q_{4}^{100}$ and its zeros for $a=0.34$ and $c=2$.}
\label{S6-figura1}
\end{figure}
As was proved in \cite[Corollary 3.4 b), p. 24]{B-JCAM95}, if $c+1<\inf I$,
then $Q_{n}^{\lambda }(x)$ has $n$\ real zeros $\{\eta _{n,r}^{\lambda
}\}_{r=1}^{n}$ satisfying the interlacing property%
\begin{equation}
\eta _{n,1}^{\lambda }<x_{n,1}<\eta _{n,2}^{\lambda }<x_{n,2}<\cdots <\eta
_{n,n}^{\lambda }<x_{n,n}  \label{S6-Interlac1}
\end{equation}%
where $I$ is the interval containing the spectrum of the Poisson
distribution $\psi ^{(a)}$, as was introduced in (\ref{S1-CharlierMeasure}).
Note that in the present case, $\inf I=0$.

The Interlacing Lemma deals only with real and simple zeros, so in order to
get the results of this Section, it will be necessary to put some extra
restriction on the values of $c$. From now on we make the assumption $c\in 
\mathbb{R}\diagdown \lbrack -1,+\infty )$, and therefore we restrict
ourselves to the case in which the set of zeros of $Q_{n}^{\lambda }(x)$ are
all simple and real.

On the other hand, from the connection formula \cite[(2.5)]{B-JCAM95}%
\begin{equation}
Q_{n}^{\lambda }(x)=C_{n}^{(a)}(x)-\frac{\lambda \Delta C_{n}^{(a)}(c)}{%
1+\lambda K_{n-1}^{(1,1)}(c,c)}K_{n-1}^{(0,1)}(x,c)  \label{S6-QconK}
\end{equation}%
we define the following \textit{limit}\ polynomials%
\begin{equation}
G_{n}^{(a)}(x)=\lim_{\lambda \rightarrow \infty }Q_{n}^{\lambda
}(x)=C_{n}^{(a)}(x)-\frac{\Delta C_{n}^{(a)}(c)}{K_{n-1}^{(1,1)}(c,c)}%
K_{n-1}^{(0,1)}(x,c),  \label{S6-GconK}
\end{equation}%
where $K_{n-1}^{(0,1)}(x,c)$ is given by (see \cite[p. 21]{B-JCAM95})%
\begin{equation}
K_{n-1}^{(0,1)}(x,c)=\frac{\mathfrak{C}_{n-1}^{(a)}(x,c)}{%
||C_{n-1}^{(a)}||^{2}}C_{n}^{(a)}(x)-\frac{\mathfrak{C}_{n}^{(a)}(x,c)}{%
||C_{n-1}^{(a)}||^{2}}C_{n-1}^{(a)}(x),  \label{S6-Knm1xc(01)}
\end{equation}%
with%
\begin{equation*}
\mathfrak{C}_{n}^{(a)}(x,c)=\frac{C_{n}^{(a)}(c)+(x-c)\Delta C_{n}^{(a)}(c)}{%
(x-c)(x-c-1)}.
\end{equation*}%
In order to prove the interlacing between the zeros of $C_{n}^{(a)}(x)$ and $%
G_{n}^{(a)}(x)$, we follow a two step process.



\begin{lemma}
\label{[S6]-Lemma9}Let $c\in \mathbb{R}\diagdown \lbrack -1,+\infty )$, and
let $\{x_{n,r}\}_{r=1}^{n}$, $\{\kappa _{n-1,r}\}_{r=1}^{n-1}$\ denote the
zeros of $C_{n}^{(a)}(x)$ and $K_{n-1}^{(0,1)}(x,c)$, respectively, all
arranged in an increasing order. Then, the zeros of $K_{n-1}^{(0,1)}(x,c)$
are all real, and the inequalities%
\begin{equation*}
x_{n,1}<x_{n-1,1}<\kappa _{n-1,1}<x_{n,2}<\cdots
<x_{n,n-1}<x_{n-1,n-1}<\kappa _{n-1,n-1}<x_{n,n+1}
\end{equation*}%
hold for every $n\in \mathbb{N}$.
\end{lemma}



\begin{proof}
The main tool here will be \cite[Corollary 1.3]{JT-ANM09} applied to (\ref%
{S6-Knm1xc(01)}). The zeros of $C_{n}^{(a)}(x)$ and $C_{n-1}^{(a)}(x)$ are
all real, and interlace on the interval $\left( 0,+\infty \right) $. Their
respective coefficients in (\ref{S6-Knm1xc(01)}), namely $\mathfrak{C}%
_{n-1}^{(a)}(x,c)||C_{n-1}^{(a)}||^{-2}$ and $\mathfrak{C}%
_{n}^{(a)}(x,c)||C_{n-1}^{(a)}||^{-2}$ are both continuous and have constant
sign on $\left( 0,+\infty \right) $.

Thus, (\ref{S6-Knm1xc(01)}) satisfies the hypothesis of \cite[Corollary 1.3]%
{JT-ANM09} and therefore the Lemma follows.
\end{proof}



\begin{lemma}
\label{[S6]-Lemma10}Let $c\in \mathbb{R}\diagdown \lbrack -1,+\infty )$, and
let $\{x_{n,r}\}_{r=1}^{n}$, $\{\kappa _{n-1,r}\}_{r=1}^{n-1}$\ and $%
\{y_{n,r}\}_{r=1}^{n}$ be the zeros of $C_{n}^{(a)}(x)$, $%
K_{n-1}^{(0,1)}(x,c)$ and $G_{n}^{(a)}(x)$, respectively, all arranged in an
increasing order. Then, the $n$\ zeros of $G_{n}^{(a)}(x)$ are all real, and
the inequalities%
\begin{equation}
y_{n,1}<x_{n,1}<\kappa _{n-1,1}<y_{n,2}<\cdots <x_{n,n-1}<\kappa
_{n-1,n-1}<y_{n,n}<x_{n,n}  \label{S6-Interlac-1}
\end{equation}%
hold for every $n\in \mathbb{N}$.
\end{lemma}



\begin{proof}
Here we use \cite[Lemma 1.1]{JT-ANM09} together with (\ref{S6-GconK}). In
the former Lemma \ref{[S6]-Lemma9} we proved that the zeros of $%
C_{n}^{(a)}(x)$ and $K_{n-1}^{(0,1)}(x,c)$ are all real, and interlace in
the interval $\left( 0,+\infty \right) $. Also, the coefficient of $%
C_{n}^{(a)}(x)$ is always constant, and $-\Delta C_{n}^{(a)}(c)\diagup
K_{n-1}^{(1,1)}(c,c)$ has constant sign (for a fixed $n$) at each of the
zeros of $C_{n}^{(a)}(x)$, so therefore the Lemma follows.
\end{proof}



Next we normalize the connection formula (\ref{S6-QconK}) in a more useful
way, in order to apply the Interlacing Lemma and obtain some results
concerning monotonicity, asymptotics, and speed of convergence for the zeros 
$\{\eta _{n,r}^{\lambda }\}_{r=1}^{n}$\ in terms of the mass $\lambda $.



\begin{proposition}
\label{[S6]-Proposition5} The polynomials in $\{\tilde{Q}_{n}^{\lambda
}\}_{n\geq 0}$ , with $\tilde{Q}_{n}^{\lambda }(x)=\lambda
_{n-1}^{c}Q_{n}^{\lambda }(x)$, can be represented as%
\begin{equation}
\tilde{Q}_{n}^{\lambda }(x)=C_{n}^{(a)}(x)+\lambda
K_{n-1}^{(1,1)}(c,c)G_{n}^{(a)}(x)  \label{S6-QnHat}
\end{equation}%
where $\lambda _{n-1}^{c}=1+\lambda K_{n-1}^{(1,1)}(c,c)$, and $%
K_{n-1}^{(1,1)}(c,c)>0$ for every $n\in \mathbb{N}$.
\end{proposition}


\begin{proof}
From (\ref{S6-GconK}) we have%
\begin{equation*}
K_{n-1}^{(0,1)}(x,c)=\frac{K_{n-1}^{(1,1)}(c,c)}{\Delta C_{n}^{(a)}(c)}\left[
C_{n}^{(a)}(x)-G_{n}^{(a)}(x)\right] .
\end{equation*}%
Combining the above expression with (\ref{S6-QconK})%
\begin{equation*}
Q_{n}^{\lambda }(x)=C_{n}^{(a)}(x)-\frac{\lambda K_{n-1}^{(1,1)}(c,c)}{%
\left( 1+\lambda K_{n-1}^{(1,1)}(c,c)\right) }\left[
C_{n}^{(a)}(x)-G_{n}^{(a)}(x)\right]
\end{equation*}%
Next we multiply the above by $1+\lambda K_{n-1}^{(1,1)}(c,c)$%
\begin{eqnarray*}
Q_{n}^{\lambda }(x)\left( 1+\lambda K_{n-1}^{(1,1)}(c,c)\right)
&=&C_{n}^{(a)}(x)\left( 1+\lambda K_{n-1}^{(1,1)}(c,c)\right) -\lambda
K_{n-1}^{(1,1)}(c,c)\left[ C_{n}^{(a)}(x)-G_{n}^{(a)}(x)\right] \\
&=&C_{n}^{(a)}(x)+\lambda K_{n-1}^{(1,1)}(c,c)G_{n}^{(a)}(x)
\end{eqnarray*}%
which yields (\ref{S6-QnHat}). Next, from (\ref{S2-S2-Kij}) one has%
\begin{equation*}
K_{n-1}^{(1,1)}(c,c)=\sum_{k=0}^{n-1}\frac{[\Delta C_{k}^{(a)}(c)]^{2}}{%
||C_{k}^{(a)}||^{2}}.
\end{equation*}%
As the right hand side of the above formula is always the sum of positive
quantities, the proof is completed.
\end{proof}



Taking into account that the positive constant $K_{n-1}^{(1,1)}(c,c)$ does
not depend on $\lambda $, we can now use (\ref{S6-QnHat}) to obtain results
about monotonicity, asymptotics, and speed of convergence for the zeros of $%
Q_{n}^{\lambda }(x)$ in terms of the mass $\lambda $. Thus, from (\ref%
{S6-QnHat}), Lemma \ref{[S6]-Lemma10}, (\ref{S6-Interlac-1}), we are in the
hypothesis of the Interlacing Lemma, and we immediately conclude the
following results.



\begin{theorem}
\label{[S6]-Theorem7} If $c\in \mathbb{R}\diagdown \lbrack -1,+\infty )$,
then the following inequalities%
\begin{equation*}
y_{n,1}<\eta _{n,1}^{\lambda }<x_{n,1}<y_{n,2}<\eta _{n,2}^{\lambda
}<x_{n,2}<\cdots <y_{n,n}<\eta _{n,n}^{\lambda }<x_{n,n}
\end{equation*}%
hold for every $n\in \mathbb{N}$. Moreover, each $\eta _{k}^{\lambda }=\eta
_{k}^{\lambda }(\lambda )$ is a decreasing function of $\lambda $ and, for
each $k=1,\ldots ,n$,%
\begin{equation}
\lim_{\lambda \rightarrow \infty }\eta _{k}^{\lambda }=y_{k}\quad \text{and}%
\quad \lim_{\lambda \rightarrow \infty }\lambda \lbrack \eta _{k}^{\lambda
}-y_{k}]=\dfrac{-C_{n}^{(a)}(y_{n,k})}{[G_{n}^{(a)}]^{\prime }(y_{n,k})}.
\label{S6-Speed}
\end{equation}
\end{theorem}



Under the above assumptions on $c$, at most one of the zeros of $%
Q_{n}^{\lambda }(x)$ is located outside $(0,+\infty )$. Next we provide the
explicit value $\lambda _{0}$ of the mass such that for $\lambda >\lambda
_{0}$ this situation appears, i.e, one of the zeros is located outside $%
(0,+\infty )$.



\begin{corollary}
\label{[S6]-Corollary2} If $c\in \mathbb{R}\diagdown \lbrack -1,+\infty )$,
then the smallest zero $\eta _{n,1}^{\lambda }=\eta _{n,1}^{\lambda }(c)$
satisfies%
\begin{equation*}
\begin{array}{c}
\eta _{n,1}^{\lambda }>0,\hspace{7pt}\mathrm{for}\hspace{7pt}\lambda
<\lambda _{0},\smallskip \\ 
\eta _{n,1}^{\lambda }=0,\hspace{7pt}\mathrm{for}\hspace{7pt}\lambda
=\lambda _{0},\smallskip \\ 
\eta _{n,1}^{\lambda }<0,\hspace{7pt}\mathrm{for}\hspace{7pt}\lambda
>\lambda _{0},%
\end{array}%
\end{equation*}%
where%
\begin{equation}
\lambda _{0}=\lambda _{0}(n,a,c)=\left( \frac{\Delta C_{n}^{(a)}(c)}{%
C_{n}^{(a)}(0)}K_{n-1}^{(0,1)}(0,c)-K_{n-1}^{(1,1)}(c,c)\right) ^{-1}>0.
\label{S6-MinMass}
\end{equation}
\end{corollary}

\begin{proof}
It suffices to use (\ref{S6-QconK}) together with the fact that $%
Q_{n}^{\lambda }\left( 0\right) =0$ if and only if $\lambda =\lambda _{0}$%
\begin{equation*}
Q_{n}^{\lambda }(0)=C_{n}^{(a)}(0)-\frac{\lambda _{0}\Delta C_{n}^{(a)}(c)}{%
1+\lambda _{0}K_{n-1}^{(1,1)}(c,c)}K_{n-1}^{(0,1)}(0,c)=0.
\end{equation*}%
Therefore%
\begin{equation*}
\lambda _{0}=\lambda _{0}(n,a,c)=\left( \frac{\Delta C_{n}^{(a)}(c)}{%
C_{n}^{(a)}(0)}K_{n-1}^{(0,1)}(0,c)-K_{n-1}^{(1,1)}(c,c)\right) ^{-1}.
\end{equation*}
\end{proof}



It would be of interest to compare the results of Theorem \ref{[S6]-Theorem7}
with \cite[Th. 4, p. 70]{DHM-NA12}. In that case, the zeros of the
Laguerre-Sobolev type polynomials also obey to an electrostatic model that
does not exist in the present Sobolev-type Charlier case. Our conjecture is
that in this case the zeros of the Sobolev-type Charlier polynomials also
seem to behave under the effect of an electrostatic potential which, so far,
is unknown to us.

Next we show some numerical experiments using Mathematica $©$ software,
dealing with the least zero of $Q_{n}^{\lambda }(x)$. We are interested to
show the location and behavior of this least zero. In the first two tables
we show the position, for some choices of the mass $\lambda $, of the first
zeros of $Q_{n}^{\lambda }(x)$ of degree $n=7$ and $a=2$. When the least
zero of the polynomial is outside $\left( 0,+\infty \right) $ is highlighted
in bold type. For $\lambda =0$ we obviously recover the least zero and the
second zero of the Charlier polynomials $C_{n}^{(a)}(x)$. When the mass
point is located at $c=-5$ we obtain%
\begin{equation*}
\begin{tabular}{|l|r|r|r|r|r|r|}
\hline
$\eta _{7,k}^{\lambda }$ & $\lambda =0$ & $\lambda =5.0\cdot 10^{-12}$ & $%
\lambda =5.0\cdot 10^{-8}$ & $\lambda =5.0\cdot 10^{-7}$ & $\lambda
=5.0\cdot 10^{-6}$ & $\lambda =5.0$ \\ \hline
$k=1$ & $0.015807$ & $0.0158059$ & $0.00424094$ & $\mathbf{-0.620631}$ & $%
\mathbf{-4.67916}$ & $\mathbf{-5.87285}$ \\ \hline
$k=2$ & $1.14616$ & $1.14616$ & $1.08515$ & $0.257578$ & $0.102767$ & $%
0.0962811$ \\ \hline
\end{tabular}%
\end{equation*}%
and for $n=10$, $a=7$, and $c=-15$ we have%
\begin{equation*}
\begin{tabular}{|l|r|r|r|r|r|r|}
\hline
$\eta _{10,k}^{\lambda }$ & $\lambda =0$ & $\lambda =5.0\cdot 10^{-15}$ & $%
\lambda =5.0\cdot 10^{-13}$ & $\lambda =5.0\cdot 10^{-12}$ & $\lambda
=5.0\cdot 10^{-7}$ & $\lambda =5.0$ \\ \hline
$k=1$ & $0.332811$ & $0.332401$ & $0.286249$ & $\mathbf{-1.34917}$ & $%
\mathbf{-17.1465}$ & $\mathbf{-17.1471}$ \\ \hline
$k=2$ & $2.05847$ & $2.05765$ & $1.96819$ & $0.983817$ & $0.632546$ & $%
0.632544$ \\ \hline
\end{tabular}%
\end{equation*}

In support of Corollary \ref{[S6]-Corollary2}, we provide the exact values
of $\lambda _{0}$ for the above two cases. From (\ref{S6-MinMass}) we see
that the smallest zero of the Sobolev-type Charlier polynomial of degree $%
n=7 $, for $a=2$ and with the mass point located at $c=-5$ is $%
\lambda_{0}=6.55003\cdot 10^{-8}\in (5.0\cdot 10^{-8},5.0\cdot 10^{-7})$, as
one can deduce from the first table. Concerning the second table we see $%
\lambda_{0}(10,7,-15)=2.1602\cdot 10^{-12}\in (5.0\cdot 10^{-13},5.0\cdot
10^{-12})$.

Finally, another interesting question is to study, for a fixed value $%
\lambda $, the behavior of zeros of Sobolev-type Charlier polynomials in
terms of the parameter $a$. Notice that, for a fixed value of $a$ we can
loose its negative zero. We show the behavior of the first two zeros to give
more information about their relative spacing. For instance, let us show the
first two zeros of the Sobolev-type Charlier polynomials of degree $n=8$,
when $\lambda =7\cdot 10^{-9}$ and the mass point is located at $c=-9$ 
\begin{equation*}
\begin{tabular}{|r|r|r|r|r|r|r|}
\hline
$\eta _{8,k}^{\lambda }:$ & $a=1$ & $a=2$ & $a=3$ & $a=4$ & $a=5$ & $a=6$ \\ 
\hline
$k=1$ & $\mathbf{-10.2156}$ & $\mathbf{-9.17105}$ & $\mathbf{-4.43974}$ & $%
\mathbf{-0.720877}$ & $0.0143978$ & $0.315444$ \\ \hline
$k=2$ & $0.00096038$ & $0.0303099$ & $0.166524$ & $0.680407$ & $1.51815$ & $%
2.12898$ \\ \hline
\end{tabular}%
\end{equation*}



\section*{Acknowledgments}



We would like to thank the anonymous referees for carefully reading the manuscript and for giving constructive comments, which substantially helped us to improve the quality of the paper. We especially thank the anonymous referee who made us notice the Remark 1.

The work of the first author was partially supported by Dirección General de Investigación Científica y Técnica, Ministerio de Economía y Competitividad of Spain, under grant MTM2015-65888-C4-2-P.




\begin{thebibliography}{99}
\bibitem{AMRR-SJMA92} M. Alfaro, F. Marcellán, M. L. Rezola, and A.
Ronveaux, \textit{On orthogonal polynomials of Sobolev type: algebraic
properties and zeros}, SIAM J. Math. Anal., \textbf{23} (1992) 737--757.

\bibitem{Alvarez03} R. Álvarez-Nodarse, \textit{Polinomios hipergeométricos
y }$q$\textit{-polinomios}, Monografías del Seminario Matemático García de
Galdeano, No. \textbf{26}, Zaragoza, 2003. In Spanish.

\bibitem{AGM-JCAM95} R. Álvarez-Nodarse, A. G. García, and F. Marcellán 
\textit{On the properties for modifications of classical orthogonal
polynomials of discrete variables}, J. Comput. Appl. Math., \textbf{65}
(1995), 3--18.

\bibitem{AGM-JDEA00} I. Área, E. Godoy, and F. Marcellán, \textit{Inner
products involving differences: the Meixner-Sobolev polynomials}, J.
Difference Equ. Appl., \textbf{6} (2000), 1--31.

\bibitem{AGMB-JCAM00} I. Área, E. Godoy, F. Marcellán, and J. J. Moreno-Balcá%
zar, \textit{Ratio and Plancherel--Rotach asymptotics for Meixner--Sobolev
orthogonal polynomials}, J. Comput. Appl. Math., \textbf{116} (1) (2000),
63--75.

\bibitem{B-JCAM95} H. Bavinck, \textit{On polynomials orthogonal with
respect to an inner product involving differences}, J. Comput. Appl. Math., 
\textbf{57} (1995), 17--27.

\bibitem{B-AA95} H. Bavinck, \textit{On polynomials orthogonal with respect
to an inner product involving differences (The general case)}, Appl. Anal., 
\textbf{59} (1995), 233--240.

\bibitem{B-IM96} H. Bavinck, \textit{A difference operator of infinite order
with the Sobolev-type Charlier polynomials as eigenfunctions}, Indag. Math., 
\textbf{7} (3), (1996), 281--291.

\bibitem{BDR-JCAM02} C. F. Bracciali, D. K. Dimitrov, and A. Sri Ranga, 
\textit{Chain sequences and symmetric generalized orthogonal polynomials},
J. Comput. Appl. Math., \textbf{143} (2002), 95--106.

\bibitem{Chi78} T. S. Chihara, \textit{An Introduction to Orthogonal
Polynomials}. Gordon and Breach, New York (1978).

\bibitem{DMR-ANM10} D. K. Dimitrov, M. V. Mello, and F. R. Rafaeli, \textit{%
Monotonicity of zeros of Jacobi-Sobolev-type orthogonal polynomials}, Appl.
Numer. Math. \textbf{60} (2010), 263--276.

\bibitem{DHM-NA12} H. Dueñas, E. J. Huertas, and F. Marcellán, \textit{%
Asymptotic properties of Laguerre-Sobolev type orthogonal polynomials},
Numer. Algorithms, \textbf{60} (1), (2012), 51--73.

\bibitem{Ism05} M. E. H. Ismail, \textit{Classical and Quantum Orthogonal
Polynomials in One Variable}. Encyclopedia of Mathematics and its
Applications Vol. \textbf{98}. Cambridge University Press. Cambridge UK.
(2005).

\bibitem{JT-ANM09} K. Jordaan and F. Toókos, \textit{Interlacing theorems
for the zeros of some orthogonal polynomials from different sequences},
Appl. Numer. Math., \textbf{59} (2009), 2015--2022.

\bibitem{KD-AAS12} S. F. Khwaja and A. B. Olde Daalhuis, \textit{Uniform
asymptotic approximations for the Meixner-Sobolev polynomials}, Anal. Appl.
(Singap.), \textbf{10} (3) (2012), 345--361.

\bibitem{Koekoek10} R. Koekoek, P. A. Lesky, and R. F. Swarttouw, \textit{%
Hypergeometric orthogonal polynomials and their q-analogues}, Springer
Monographs in Mathematics. Springer-Verlag, Berlin (2010).

\bibitem{MPP-RdM92} F. Marcellán, T. E. Pérez, and M. A. Piñar, \textit{On
zeros of Sobolev-type orthogonal polynomials}, Rend. Mat. Appl., \textbf{12}
(7) (1992), 455--473.

\bibitem{MX-EM14} F. Marcellán and Y. Xu, \textit{On Sobolev orthogonal
polynomials}, Expo. Math., \textbf{33} (2015), 308--352.

\bibitem{M-JCAM93} H. G. Meijer, \textit{On real and complex zeros of
orthogonal polynomials in a discrete Sobolev space}, J. Comput. Appl. Math., 
\textbf{49} (1993) 179--191.

\bibitem{M-JCAM15} J. J. Moreno-Balcázar, \textit{$\Delta$-Meixner-Sobolev
orthogonal polynomials: Mehler--Heine type formula and zeros}, J. Comput.
Appl. Math., \textbf{284} (2015), 228--234.

\bibitem{MPP-RJ11} J. J. Moreno-Balcázar, T. E. Pérez, and M. A. Piñar, 
\textit{A generating function for non-standard orthogonal polynomials
involving differences: the Meixner case}, Ramanujan J., \textbf{25} (2011),
21--35.

\bibitem{Nikiforov91} A. F. Nikiforov, S. K. Suslov, and V. B. Uvarov, 
\textit{Classical Orthogonal Polynomials of a Discrete Variable}, Springer
Series in Computational Physics. Springer-Verlag, Berlin, Heidelberg (1991).

\bibitem{Szego75} G. Szeg\H{o}, \textit{Orthogonal Polynomials}. $4^{th}$
ed., Amer. Math. Soc. Colloq. Publ. Series Vol \textbf{23} Amer. Math. Soc.
Providence, RI. (1975).

\end{thebibliography}
\end{document}